%% file: 1.tex
\documentclass[12pt]{article}
\usepackage{graphicx}
\usepackage{amsmath}
\usepackage{amsfonts}
\usepackage{amsthm}
\usepackage[T1]{fontenc}
\usepackage{url}
\usepackage{color}
\usepackage[margin=1in]{geometry}
\input{mymacros}

\date{}
\begin{document}
\title{Rigidity of a thin domain depends on the curvature, width, and boundary conditions}
\author{Zh. Avetisyan\thanks{University of California Santa Barbara, harutyunyan@math.ucsb.edu}, D. Harutyunyan\thanks{University of California Santa Barbara, z.avetisyan@math.ucsb.edu},
and N. Hovsepyan\thanks{Temple University, narek.hovsepyan@temple.edu}}

\maketitle

\begin{abstract}
This paper is concerned with the study of linear geometric rigidity of shallow thin domains under zero Dirichlet boundary conditions on the displacement field on the thin edge of the domain. 
A shallow thin domain is a thin domain that has in-plane dimensions of order $O(1)$ and $\epsilon,$ where $\epsilon\in (h,1)$ is a parameter (here $h$ is the thickness of the shell). The problem has been solved in [8,10] for the case $\epsilon=1,$ with the outcome of the optimal constant $C\sim h^{-3/2},$ $C\sim h^{-4/3},$ and $C\sim h^{-1}$ for parabolic, hyperbolic and elliptic thin domains respectively. We prove in the present work that in fact there are two distinctive scaling regimes $\epsilon\in (h,\sqrt h]$ and $\epsilon\in (\sqrt h,1),$ such that in each of which the thin domain rigidity is given by a certain formula in $h$ and $\epsilon.$ An interesting new phenomenon is that in the first (small parameter) regime $\epsilon\in (h,\sqrt h]$, the rigidity does not depend on the curvature of the thin domain mid-surface. 
  
\end{abstract}

\tableofcontents

\section{Introduction}
\label{sec:1}
Given a small parameter $h>0$, a shell of thickness $h$ in the theory of elasticity is the $h-$neighborhood (in the normal direction) of the a smooth enough connected and compact 
surface $S\subset\mathbb R^3.$ Denoting the normal vector (note that the orientation does not matter here) by $\Bn(x)$ for any point $x\in S,$ and assuming that 
the family of Lipschitz functions $g_1^h(x),g_2^h(x)\colon S\to (0,\infty)$ satisfy the uniform conditions
\begin{equation}
\label{1.1}
h\leq g_1^h(x),g_2^h(x)\leq c_1 h,\quad \text{and}\quad |\nabla g_1^h(x)|+|\nabla g_2^h(x)|\leq c_2h,\quad\text{for all}\quad x\in S,
\end{equation}
one defines a thin domain around $S$ and having thickness of order $h,$ to be the set $\Omega^h=\{x+t\Bn(x) \ : \ x\in S,\ t\in (-g_1^h(x),g_2^h(x))\}.$
The problem of determining the rigidity of a given shell or thin domain is one of the main challenges in the theories of linear and nonlinear elasticity. A mathematical definition of the geometric rigidity of a shell is formulated by means of the the geometric rigidity estimate of Friesecke, James and M\"uller [\ref{bib:Fri.Jam.Mue.1},\ref{bib:Fri.Jam.Mue.2}], which reads as follows: 
\textit{Let $\Omega\subset\mathbb R^3$ be open bounded connected and Lipschitz. There exists a constant $C=C(\Omega)$ such that for every vector field $\Bu\in H^1(\Omega)$ there exists a constant rotation $\BR\in SO(3)$, such that}
\begin{equation}
\label{1.2}
\|\nabla\Bu-\BR\|_{L^2(\Omega)}^2\leq C\int_\Omega\mathrm{dist}^2(\nabla\Bu(x),SO(3))dx.
\end{equation}
It is a well-known fact in elasticity theory that for thin domains $\Omega,$ the constant $C$ in (\ref{1.2}) blows up as the thickness $h$ of the domain $\Omega$ goes to zero. The optimal 
constant in (\ref{1.2}) then determines the rigidity of $\Omega,$ the bigger the constant $C$ is the less rigid the domain $\Omega$ is. It turns out that this mathematical definition goes well 
hand-in-hand with the physical understanding of rigidity, namely, typically mathematically more rigid thin domains handle higher compressive loads before buckling. 
The form (\ref{1.2}) comes from nonlinear elasticity, where the energy well is the group of proper rotations $SO(3).$ It is usually assumed in the study of shells in nonlinear elasticity
that the elastic energy density $W$ satisfies the lower bound $W(\BF)\geq c\cdot\mathrm{dist}^2(\BF,SO(3))$ for some $c>0$ and all $\BF\in\mathbb R^{3\times 3}.$ If $\Omega^h$ is a family of thin domains of thickness of order $h,$ the for low-energy\footnote{Low means asymptotically small as $h\to 0$.} (small enough energies) displacements $\Bu^h,$ the estimate (\ref{1.2}) will imply relative weak compactness of the displacement family $\{\Bu^h\}$ in $H^1,$ and strong compactness of a subsequence of $\{\Bu^h\}$ in $L^2.$ This being said, the derivation of shell theories from three dimensional elasticity is basically based on the geometric rigidity estimate (\ref{1.2}). Hence, in that analysis the asymptotics of the optimal constant $C$ in (\ref{1.2}) becomes crucial, as it identifies the energy scaling ranges (in terms of $h$) for the dimension reduction theory to hold. This is addressed as the shell rigidity problem. As the tangent space to the manifold $SO(3)$ at the identity matrix $\BI$ is the vector space of skew-symmetric matrices, the linearization of (\ref{1.2}) around $\BI$ is exactly classical Korn's first inequality without boundary conditions [\ref{bib:Korn.1},\ref{bib:Korn.2},\ref{bib:Friedrichs},\ref{bib:Kohn}], which reads as follows: \textit{Let $\Omega\subset\mathbb R^n$ be open bounded connected and Lipschitz. There exists a constant $C_1=C_1(\Omega),$ depending only on $\Omega,$ such that for every vector field $\Bu\in H^1(\Omega)$ there exists a skew-symmetric matrix $\BA\in \mathbb R^{n\times n}$ such that 
\begin{equation}
\label{1.3}
\|\nabla\Bu-\BA\|_{L^2(\Omega)}^2\leq C_1\|e(\Bu)\|_{L^2(\Omega)}^2,
\end{equation}
where $e(\Bu)=\frac{1}{2}(\nabla\Bu+\nabla\Bu^T)$ is the symmetric part of the gradient (the strain in linear elasticity).} 
It is known that if in addition one imposes Dirichlet type boundary condition on the displacement field $\Bu$ on some positive Hausdorff $H^2$-measure part of the 
boundary $\partial\Omega$ then one can choose $\BA=0$ in (\ref{1.3}). This version of (\ref{1.3}) plays a crucial role in the study of critical buckling loads for shells 
(and thin structures in general) under compression [\ref{bib:Gra.Tru.},\ref{bib:Gra.Har.1},\ref{bib:Gra.Har.2},\ref{bib:Gra.Har.3}]. In particular, the authors in [\ref{bib:Gra.Har.2},\ref{bib:Gra.Har.3}] 
prove a buckling load formula for slender structures under compression under some stability conditions on the deformation, where the formula involves the "best"\footnote{"Best" meaning asymptocically optimal as $h\to 0$.} constant $C_1$ in (\ref{1.3}), see [\ref{bib:Gra.Har.3}, Theorem~2.6]. In this work we will be studying (\ref{1.3}), namely the asymptotics of the optimal 
constant $C_1$ under zero Dirichlet boundary conditions on $\Bu$ over the thin edge of $\partial\Omega.$ This problem has been solved in [\ref{bib:Gra.Har.4}] and [\ref{bib:Harutyunyan.2}] 
for parabolic and elliptic, and hyperbolic shells respectively, in the case when the thin domain $\Omega$ has in-plane dimensions of order one in booth principal directions on the mid-surface $S.$
In this paper we will consider the case when $S$ has width of order one in one principal direction $\Gth,$ and of order $\epsilon>0$ in the other principal direction $z,$ where 
$\epsilon\in[h,1]$ is a parameter. We call this kind of thin domains shallow. Interestingly enough, we discover that there are two scaling regimes with a crossover $\tilde h=\sqrt{h}$ distinguishing between the formulae for the "best" constant $C_1$ in (\ref{1.3}) for all three kinds of shells, parabolic, hyperbolic, and elliptic. In the first regime 
$\epsilon\in [h,\sqrt{h}]$ the formula for $C_1$ does not depend on the Gaussian curvature of $S$ and is given by $C_1\sim \epsilon^2h^{-2},$ see Theorem~\ref{th:2.1}.  In the second regime 
$\epsilon\in [\sqrt{h},1]$ the "best" constant $C_1$ does in contrast depend on the Gaussian curvature of $S$ and is given by $C_1\sim \epsilon h^{-3/2},$ $C_1\sim \epsilon^{4/3}h^{-4/3},$ and 
$C_1\sim h^{-1}$ for parabolic, hyperbolic and elliptic shells respectively, see Theorem~\ref{th:2.2}. These recover the results in  [\ref{bib:Gra.Har.4},\ref{bib:Harutyunyan.2}] taking $\epsilon=1.$ 
Another interesting physical phenomenon is that in the case of elliptic shells, the entire range $\epsilon\in [\sqrt{h},1]$ gives the same scaling $C_1\sim h^{-1}.$ Also, we will prove that 
when the zero boundary conditions are removed, the in the regime $\epsilon\in [h,\sqrt{h}]$ the rigidity drops to $h^{-2}$ in some cases. \textit{This observation 
reveals the (explicit) dependence of the rigidity on the size, curvature, and boundary conditions. }

\section{Main results}
\setcounter{equation}{0}
\label{sec:2}

In this section we will introduce the main notation, definitions, and formulate the main results of the paper. First of all we will assume that the mid-surface $S\subset \mathbb R^3$ is a compact, connected, regular $C^3$ surface with nonempty relative interior, that can be given by a single patch parametrization $\Br=\Br(\Gth,z)$ in the principal variables $\Gth$ and $z,$ where we assume $\Gth\in [0,1]$ and $z\in [z_1(\Gth),z_2(\Gth)]$ for every $\Gth\in [0,1].$ 
Let $E$ denote the patch: $E=\{(\Gth,z) \ : \ \Gth\in[0,1] \ z\in [z_1(\Gth),z_2(\Gth)]\}.$ In order for $S$ to have roughly size $\epsilon$ in the $z-$direction, we impose the below condition on $z_1(\Gth)$ and $z_2(\Gth):$
\begin{equation}
\label{2.1}
0\leq z_1(\Gth),\quad \epsilon\leq z_2(\Gth)-z_1(\Gth)\leq c_3\epsilon\quad\text{for all}\quad \Gth\in [0,1],
\end{equation}
for some constant $c_3\geq 1.$ Let now $\Bn(\Gth,z)$ be the unit normal\footnote{The surface $S$ need not be orientable hare, as the choice of $\Bn$ or $-\Bn$ does not affect the presentation.} to $S$ at $(\Gth,z),$ and let $t$ denote the normal variable. One then naturally has a parametrization of the thin domain $\Omega^{h,\epsilon}$ given by $\BR(t,\Gth,z)=\Br(z,\Gth)+t\Bn(z,\Gth),$ 
where $(\Gth,z)\in E$ and $t\in [-g_1^h(\Gth,z),g_2^h(\Gth,z)],$ where $g_1^h(\Gth,z)$ and  $g_2^h(\Gth,z)$ are given as right before (\ref{1.1}) and fulfill (\ref{1.1}). 
Denote $A_{z}=\left|\frac{\partial \Br}{\partial z}\right|,$ $A_{\Gth}=\left|\frac{\partial \Br}{\partial\Gth}\right|$ 
the two nonzero (as $S$ is regular) components of the metric tensor on $S$ and denote $\Gk_{z}$ and $\Gk_{\Gth}$ the two principal curvatures.
Recall that in the orthonormal basis $(\Bn,\Be_\Gth,\Be_z)$ the gradient of any vector field $\Bu=(u_t,u_\Gth,u_z)\in H^1(\Omega^{h,\epsilon},\mathbb R^3)$ on the thin domain 
$\Omega^{h,\epsilon}$ is given by the formula [\ref{bib:Tov.Smi.}],
\begin{equation}
\label{2.2}
\nabla\Bu=
\begin{bmatrix}
  u_{t,t} & \dfrac{u_{t,\Gth}-A_{\Gth}\Gk_{\Gth}u_{\Gth}}{A_{\Gth}(1+t\Gk_{\Gth})} &
\dfrac{u_{t,z}-A_{z}\Gk_{z}u_{z}}{A_{z}(1+t\Gk_{z})}\\[3ex]
u_{\Gth,t}  &
\dfrac{A_{z}u_{\Gth,\Gth}+A_{z}A_{\Gth}\Gk_{\Gth}u_{t}+A_{\Gth,z}u_{z}}{A_{z}A_{\Gth}(1+t\Gk_{\Gth})} &
\dfrac{A_{\Gth}u_{\Gth,z}-A_{z,\Gth}u_{z}}{A_{z}A_{\Gth}(1+t\Gk_{z})}\\[3ex]
u_{ z,t}  & \dfrac{A_{z}u_{z,\Gth}-A_{\Gth,z}u_{\Gth}}{A_{z}A_{\Gth}(1+t\Gk_{\Gth})} &
\dfrac{A_{\Gth}u_{z,z}+A_{z}A_{\Gth}\Gk_{z}u_{t}+A_{z,\Gth}u_{\Gth}}{A_{z}A_{\Gth}(1+t\Gk_{z})}
\end{bmatrix}
\end{equation}
The norm $\|f\|_{L^2(\Omega^{h,\epsilon})}$ then comes from the inner product of two functions
$f,g\colon\Omega^{h,\epsilon}\to\mathbb R,$ given by $(f,g)_{\Omega^{h,\epsilon}}=\int_{\Omega^{h,\epsilon}}A_zA_\Gth f(t,\Gth,z)g(t,\Gth,z)d\Gth dzdt$
(recall that $A_\Gth$ and $A_z$ are strictly positive and thus apart from zero on $E$). Following [\ref{bib:Harutyunyan.2},\ref{bib:Harutyunyan.3},\ref{bib:Harutyunyan.4}], we introduce the thin domain $O(1)$ parameters that are the quantities $c_1,c_2,c_3,a,A,k,$ and $K,$ some of which defined below (note that $K$ is defined later in Theorem~\ref{th:2.2}) satisfy the conditions
\begin{align}
\label{2.3}
&a=\min_{E}(A_\Gth,A_z)>0, \quad A=\|A_\Gth\|_{W^{2,\infty}(E)}+\|A_z\|_{W^{2,\infty}(E)}<\infty,\\ \nonumber
&k=\|\Gk_\Gth\|_{W^{1,\infty}(E)}+\|\Gk_z\|_{W^{1,\infty}(E)}<\infty.
\end{align}
due to the fact that $C$ is regular, compact and of class $C^3.$ The constants $c_1,c_2$ are introduced in (\ref{1.1}) and $c_3$ is introduced in (\ref{2.1}). The constants $h_0>0$ and $C>0$ in the below theorems will depend only on the quantities $a,A,k,c_1,c_2,c_3$ and the constant $K$ defined in Theorem~\ref{th:2.2}. Finally, we introduce the vector space $V^{h,\epsilon}$ of displacements satisfying zero Dirichlet boundary condition on the thin part of the boundary of $\Omega^{h,\epsilon}.$ 
To that end denote that thin part 
$$\partial_S\Omega^{h,\epsilon}=\{(t,\Gth,z)\in\partial\Omega^{h,\epsilon} \ : \ \Gth(\Gth-1)(z-z_1(\Gth))(z-z_2(\Gth))=0\}.$$
Then the vector space is 
\begin{equation}
\label{2.4}
V^{h,\epsilon}=\{\Bu\in H^1(\Omega^{h,\epsilon},\mathbb R^3) \ : \  \Bu(t,\Gth,z)=0 \ \ \text{for}\ \ (t,\Gth,z)\in\partial_S\Omega^{h,\epsilon}\}. 
\end{equation}

We are ready to formulate the main results of the paper. 

\begin{theorem}[Korn's first inequality for shallow thin domains]
\label{th:2.1}
Let $S\subset\mathbb R^3$ be a connected, compact, regular $C^3$ surface with nonempty relative interior satisfying (\ref{2.1}) (with the parametrization $\Br=\Br(\Gth,z), (\Gth,z)\in E$), and let the thin domain $\Omega^{h,\epsilon}$ around $S$ be given as in the next two lines of (\ref{1.1}) with the barrier functions $g_1^h$ and $g_2^h$ satisfying (\ref{1.1}). If $\epsilon\in [h,\sqrt{h}],$ then there exists constants $h_0,C>0,$ depending only on the thin domain $O(1)$ parameters, such that Korn's first inequality holds (regardless of the Gaussian curvature sign):
\begin{equation}
 \label{2.5}
\|\nabla\Bu\|^2_{L^2(\Omega^{h,\epsilon})}\leq \frac{C\epsilon^2}{h^2}\|e(\Bu)\|^2_{L^2(\Omega^{h,\epsilon})},
\end{equation}
for all $h\in(0,h_0)$ and $\Bu\in V^{h,\epsilon}.$ Moreover, the constant in (\ref{2.5}) is optimal, i.e, there exists a sequence of displacements 
$\Bu^{h,\epsilon}\in V^{h,\epsilon}$ realizing the asymptotics of $h$ and $\epsilon$ in (\ref{2.5}) as $h,\epsilon\to 0.$ 
\end{theorem}

The next theorem is concerned with the second regime.  

\begin{theorem}[Korn's first inequality]
\label{th:2.2}
Let $S\subset\mathbb R^3$ be a connected, compact, regular $C^3$ surface with nonempty relative interior satisfying (\ref{2.1}) 
(with the parametrization $\Br=\Br(\Gth,z), (\Gth,z)\in E$), and let the thin domain $\Omega^{h,\epsilon}$ around $S$ be given as in the next two
lines of (\ref{1.1}) with the barrier functions $g_1^h$ and $g_2^h$ satisfying (\ref{1.1}). Denote the Gaussian curvature of $S$ by $K_G=\Gk_\Gth \Gk_z.$ If $\epsilon\in [\sqrt{h},1],$ 
then there exists constants $h_0,C>0$ depending only on the thin domain $O(1)$ parameters, such that in each of the situations Korn's first inequality holds:
\begin{itemize}
\item[1(a).] If $K_G>0,$ then then there exists a constant $\epsilon_0>0,$ depending only on the parameters of $S,$ such that
\begin{equation}
 \label{2.6}
\|\nabla\Bu\|^2_{L^2(\Omega^{h,\epsilon})}\leq \frac{C}{h}\|e(\Bu)\|^2_{L^2(\Omega^{h,\epsilon})},
\end{equation}
for all $h\in (0,h_0),$ all $\epsilon\in (h,\epsilon_0),$ and all $\Bu\in V^{h,\epsilon}.$ Here $K=\min_{E}(|\Gk_\Gth|,|\Gk_z|).$
\item[1(b).] If $K_G<0,$
\begin{equation}
 \label{2.7}
\|\nabla\Bu\|^2_{L^2(\Omega^{h,\epsilon})}\leq \frac{C\epsilon^{2/3}}{h^{4/3}}\|e(\Bu)\|^2_{L^2(\Omega^{h,\epsilon})},
\end{equation}
for all $h\in (0,h_0),$ and all $\Bu\in V^{h,\epsilon}.$ Here $K=\min_{E}(|\Gk_\Gth|,|\Gk_z|).$
\item[1(c).] If $K_G=0,$ with $\Gk_z=0$ and $K=\min_{E}|\Gk_\Gth|>0,$ then
\begin{equation}
 \label{2.8}
\|\nabla\Bu\|^2_{L^2(\Omega^{h,\epsilon})}\leq \frac{C\epsilon}{h^{3/2}}\|e(\Bu)\|^2_{L^2(\Omega^{h,\epsilon})},
\end{equation}
for all $h\in (0,h_0),$ and all $\Bu\in V^{h,\epsilon}.$ 
\end{itemize}
Moreover the constants in (\ref{2.6})-(\ref{2.8}) are optimal, i.e, there exists sequences of displacements 
$\Bu^{h,\epsilon}\in V^{h,\epsilon}$ (in each situation) realizing the asymptotics of $h$ and $\epsilon$ in (\ref{2.6})-(\ref{2.8}) as $h,\epsilon\to 0.$ 

\end{theorem}

\begin{remark}
\label{rem:2.3}
Both Theorems 2.1 and 2.2 clearly reveal the dependence of the optimal constant in Korn's first inequality on the thin domain width. In fact, in the absence of zero Dirichlet boundary conditions on the thin part of $\partial\Omega,$ the optimal constant in both (\ref{1.3}) and (\ref{1.2}) drop to $h^{-2}$ in the shallow domain regime $\epsilon\in[h,\sqrt{h}].$ This will be addresses later 
in Section~5.
\end{remark}

\section{Preliminary}
\setcounter{equation}{0}
\label{sec:3}

In the proof of Theorems~2.1 and 2.2 we adopt the overall strategies developed in [\ref{bib:Gra.Har.1},\ref{bib:Harutyunyan.1},\ref{bib:Gra.Har.4},\ref{bib:Harutyunyan.2}]. 
At its core is the reduction of $3D$ inequalities to $2D$ ones, which are then proven by the use of several key lemmas that provide inequalities with sharp constants 
for harmonic functions in thin domains. A $3D$ estimate that already plays a key role in the business of sharp Korn's inequality is the so-called Korn interpolation inequality for shells, introduced 
first for cylindrical shells in [\ref{bib:Gra.Har.1}]. Another simple but important component of the analysis is the fact that if one has a Korn's first inequality (\ref{1.3}) or a geometric 
rigidity estimate (\ref{1.2}) in the shell $S^{h,\epsilon}$ with thickness $h$ around $S$ defined by 
\begin{equation}
 \label{3.1}
 S^{h,\epsilon}=\{(t,\Gth,z)\in\mathbb R^3\ : \ (\Gth,z)\in E, \  t\in (-h,h)\},
\end{equation}
then in fact the same inequality with a comparable constant to $C_1$ (or to $C$) follows for the thin domain $\Omega^{h,\epsilon}.$ This passage is carried out via a localization argument 
by Kohn and Vogelius [\ref{bib:Koh.Vog.}]; the detailes will be provided in Section~3.4. We formulate below the Korn interpolation inequality. 
\begin{lemma}
\label{lem:3.1}
Let the surface $S\subset\mathbb R^3$ be as in Theorem~2.1, and let the shell $S^{h,\epsilon}$ and the vector space $V^{h,\epsilon}$ be defined as in (\ref{3.1}) and 
(\ref{2.4}) respectively. Then there exists constant $C,h_0,$ depending only on $a,A,k$ and $c_3,$ such that
\begin{equation}
 \label{3.2}
\|\nabla\Bu\|_{L^2(S^{h,\epsilon})}^2\leq C\left(\frac{\|u_t\|_{L^2(S^{h,\epsilon})}\|e(\Bu)\|_{L^2(S^{h,\epsilon})}}{h}+\|\Bu\|_{L^2(S^{h,\epsilon})}^2+\|e(\Bu)\|_{L^2(S^{h,\epsilon})}^2\right),
\end{equation} 
for all $\Bu\in V^{h,\epsilon},$ all $h\in (0,h_0),$ and all $\epsilon\in[h,1].$ 
\end{lemma}

It is worth mentioning that slightly stronger versions inequality (\ref{3.2}) have been proven in [\ref{bib:Gra.Har.4}, Theorem~3.1] and in [\ref{bib:Harutyunyan.2}, Theorem~3.2] for the case 
$\epsilon=O(1).$ Also, the same version has been proven in [\ref{bib:Harutyunyan.3}, Theorem~3.1] (see also [\ref{bib:Harutyunyan.4}]) without assuming any boundary condition on $\Bu.$ 
Note that after dividing $S^{h,\epsilon}$ into roughly $1/\epsilon$ parts in the $\Gth$ direction, variable rescale $x'=\epsilon x$ and application 
of (\ref{3.2}) (with no boundary conditions on $\Bu$) one would deduce the estimate 
\begin{equation}
 \label{3.3}
 \|\nabla\Bu\|_{L^2(S^{h,\epsilon})}^2\leq C\left(\frac{\|u_t\|_{L^2(S^{h,\epsilon})}\|e(\Bu)\|_{L^2(S^{h,\epsilon})}}{h}+\frac{\|\Bu\|_{L^2(S^{h,\epsilon})}^2}{\epsilon^2}
 +\|e(\Bu)\|_{L^2(S^{h,\epsilon})}^2\right),
\end{equation}
 for all $\Bu\in H^1(S^{h,\epsilon},\mathbb R^3).$ It will be seen later that (\ref{3.3}) is not good enough for the purpose of Theorems 2.1 and 2.2 for the full range $\epsilon\in[h,1]$ and all cases
 (\ref{2.5})-(\ref{2.8}). This being said we will need the validity of (\ref{3.2}) for the full range $\epsilon\in[h,1].$ As already mentioned, the structure of the current paper will be quite 
 similar to the ones in [\ref{bib:Gra.Har.4},\ref{bib:Harutyunyan.2}], thus we will skip some straightforward steps in the proof referring to that previous works, and at the same time 
 trying to keep the current paper as self-contained as possible without major repetitions. Note that both Theorems 2.1 and 2.2 claim Ansatz-free lower bounds and their sharpness, where the sharpness part will be proven by providing an Ansatz that makes the inequality an asymptotic equality as $h,\epsilon\to 0.$ We start with the Ansatz-free lower bounds past.

\subsection{Ansatz-free lower bounds: Interpolation inequalities}
\label{sub:3.1}

In the sequel the $\|\cdot \|$ norm will be the $L^2$ norm $\|\cdot \|_{L^2(S^{h,\epsilon})},$ and the constants $C,h_0$ will depend only on $a,A,k,$ and $c_3$ 
throughout Section~3 unless specified otherwise. The proof is performed by sequential freeze of the acting variables $t,\Gth,z$ to reduce the $3D$ inequality of the corresponding block 
of the gradient $\Bu$ to a $2D$ inequality. Before doing that, we make the following (again adopted in [\ref{bib:Gra.Har.4},\ref{bib:Harutyunyan.2}]) 
simplification to $\nabla\Bu$ by removing the $t\Gk_\Gth$ and $t\Gk_z$ parts in the denominators of the second and third columns respectively. 
 Namely, setting 
 \begin{equation}
\label{3.4}
\BB=
\begin{bmatrix}
  u_{t,t} & \dfrac{u_{t,\Gth}-A_{\Gth}\Gk_{\Gth}u_{\Gth}}{A_{\Gth}} & \dfrac{u_{t,z}-A_{z}\Gk_{z}u_{z}}{A_{z}}\\[3ex]
u_{\Gth,t}  & \dfrac{A_{z}u_{\Gth,\Gth}+A_{z}A_{\Gth}\Gk_{\Gth}u_{t}+A_{\Gth,z}u_{z}}{A_{z}A_{\Gth}} & \dfrac{A_{\Gth}u_{\Gth,z}-A_{z,\Gth}u_{z}}{A_{z}A_{\Gth}}\\[3ex]
u_{ z,t}  & \dfrac{A_{z}u_{z,\Gth}-A_{\Gth,z}u_{\Gth}}{A_{z}A_{\Gth}} & \dfrac{A_{\Gth}u_{z,z}+A_{z}A_{\Gth}\Gk_{z}u_{t}+A_{z,\Gth}u_{\Gth}}{A_{z}A_{\Gth}}
\end{bmatrix}
\end{equation} 
 we first aim at the simplified version of (\ref{3.2}), which reads as 
 \begin{equation}
 \label{3.5}
\|\BB\|^2\leq C\left(\frac{\|u_t\| \|\BB^{sym}\|}{h}+\|\Bu\|^2+\|\BB^{sym}\|^2\right),
\end{equation} 
where $\BB^{sym}=\frac{1}{2}(\BB+\BB^T)$ is the symmetric part of the matrix $\BB.$ The later passage from (\ref{3.5}) to (\ref{3.2}) will then be quite straightforward taking into account the 
smallness of $t$ and thus the omitted terms. 

\begin{proof}[Proof of (\ref{3.5})] 
Consider the blocks $t=const,$ $\Gth=const,$ and $z=const$ separately. \\
 \textbf{The block $t=const.$} \textit{There exists a constant $C>0$ such that 
 \begin{equation}
 \label{3.6}
\|B_{23}\|^2+\|B_{32}\|^2\leq C(\|\Bu\|^2+\|\BB^{sym}\|^2),
\end{equation}   
for all $h\in (0,1)$ and $\Bu\in V^{h,\epsilon}$.} 
\begin{proof}
The analysis for this block is the simplest. For functions $f,g\in L^2(E,\mathbb R)$ define the inner product 
 $(f,g)_E=\int_E A_\Gth A_zfgd\Gth dz.$ Hence we have on one hand by Fubini's theorem and integrating by parts, 
 \begin{align*}
 \left(B_{23}+\frac{A_{z,\Gth }u_z}{A_\Gth A_z}, B_{32}+\frac{A_{\Gth,z} u_\Gth}{A_\Gth A_z}\right)_E &=\int_E u_{\Gth,z}u_{z,\Gth} \\
 &=\int_E u_{\Gth,\Gth}u_{z,z}.
 \end{align*}
 On the other hand we have 
 \begin{align*}
\int_E u_{\Gth,\Gth}u_{z,z}&=\left( B_{22}-\Gk_\Gth u_t-\frac{A_{\Gth,z }u_z}{A_\Gth A_z}, B_{33}-\Gk_z u_t-\frac{A_{z,\Gth} u_\Gth}{A_\Gth A_z}\right)_E, 
  \end{align*}
 thus recalling the equalities $B_{22}=B_{22}^{sym},$ $B_{33}=B_{33}^{sym},$ and $B_{32}=2B_{23}^{sym}-B_{23},$ we obtain from the last two equalities the bound 
$$ |(B_{23},B_{23})_E|\leq C(|(\BB^{sym},\BB^{sym})_E|+|(\Bu,\Bu)_E|),$$
which then integrating in $t\in (-h,h)$ we discover 
 \begin{equation}
 \label{3.7}
\|B_{23}\|^2\leq C(\|\Bu\|^2+\|\BB^{sym}\|^2).
\end{equation}   
Consequently taking into account one more time the equality $B_{32}=2B_{23}^{sym}-B_{23},$ we obtain the estimate (\ref{3.6}) from (\ref{3.7}) by the triangle inequality. 

\end{proof}

\noindent\textbf{The block $\Gth=const.$} \textit{There exists a constant $C>0,$ such that one has the estimate 
\begin{equation}
 \label{3.8}
\|B_{13}\|^2+\|B_{31}\|^2\leq C\left(\frac{\|u_t\|\|\BB^{sym}\|}{h}+\frac{1}{\delta}\|\Bu\|^2+\delta\|B_{12}\|^2+\|\BB^{sym}\|^2\right).
\end{equation}   
for all $h\in (0,1),$ all $\Bu\in V^{h,\epsilon},$ and all $\delta\in (0,1)$.}
 
\begin{proof}[Proof of (\ref{3.8})] The proof for this block is based on several lemmas proven in [\ref{bib:Gra.Har.1},\ref{bib:Harutyunyan.1},\ref{bib:Harutyunyan.3}], and their modifications that we will formulate and prove when necessary. A main lemma is Lemma~5.1 from [\ref{bib:Harutyunyan.3}], which is not directly applicable to the situation $\epsilon<<1$ that falls 
into the full range $[h,1].$ This is due to the fact that the constants in the targeted inequality may depend on $\epsilon,$ which may mishandle the final constants we get. 
Therefore, we aim to prove the following modification applicable to all values of $\epsilon\in [h,1].$ 
\begin{lemma}
\label{lem:3.2}
For parameters $0<h\leq \epsilon\leq 1$ denote the two dimensional rectangle $R=(0,h)\times (0,\epsilon).$ Given a displacement field
$\BU=(u(x,y),v(x,y))\in H^1(R,\mathbb R^2),$ vector fields $\BGa,\BGb\in W^{1,\infty}(R,\mathbb R^2),$ and a function $w\in H^1(R,\mathbb R),$ denote
\begin{equation}
\label{3.9}
\BM=
\begin{bmatrix}
\partial_xu & \partial_yu+\BGa\cdot\BU\\
\partial_xv & \partial_yv+\BGb\cdot\BU+w
\end{bmatrix}.
\end{equation}
Then for any displacement field $\BU\in H^1(R,\mathbb R^2)$ satisfying the Dirichlet boundary conditions 
$$\BU(x,0)=\BU(x,\epsilon)=0,\quad \text{for all}\quad x\in(0,h)$$
in the trace sense, the following interpolation inequality holds:
\begin{align}
\label{3.10}
\|\BM\|_{L^2(R)}^2 &\leq \frac{\tilde C}{h} \|u\|_{L^2(R)}\cdot \|\BM^{sym}\|_{L^2(R)}+\|\BM^{sym}\|_{L^2(R)}^2\\ \nonumber
&+\left(\frac{1}{\delta}+h^2\right)\|\BU\|_{L^2(R)}^2+(\delta+h^2)(\|w\|_{L^2(R)}^2+\|\partial_x w\|_{L^2(R)}^2),
\end{align}
for all $\delta \in (0,1),$ $h\in (0,\tilde h).$ Here the constants $\tilde C,\tilde h>0$ (the existence of which is claimed) depend only on the quantities $\|\BGa\|_{W^{1,\infty}(R)},$ $\|\BGb\|_{W^{1,\infty}(R)}.$ 
\end{lemma}

\begin{proof}[Proof of Lemma~3.2] We will adopt the method of harmonic projections, following the proof of the similar lemma in [\ref{bib:Harutyunyan.3}] with necessary modifications. In the proof of the lemma the constant $\tilde C$ may depend only on the norms $\|\BGa\|_{W^{1,\infty}}$ $\|\BGb\|_{W^{1,\infty}}.$ Also, for simplicity we will write $\|\cdot\|$ instead of $\|\cdot\|_{L^2(R)}.$ 
By density we can assume without loss of generality that $\BU\in C^2(\bar R).$ For functions $\varphi,\phi\in H^1(R,\mathbb R)$ introduce the perturbed gradient
\begin{equation}
\label{3.11}
\BM_{\varphi,\psi}=
\nabla \BU+
\begin{bmatrix}
0 & \varphi\\
0 & \psi
\end{bmatrix}
=
\begin{bmatrix}
\partial_x u & \partial_y u+\varphi\\
\partial_x v & \partial_y v+\psi
\end{bmatrix}.
\end{equation}
Let the function $u_1$ be the unique solution to the Dirichlet boundary value problem
\begin{equation}
\label{3.12}
\begin{cases}
\Delta u_1=0 & in \ \ R\\
u_1=u & on \ \ \partial R.
\end{cases}
\end{equation}
First of all Poincar\'e inequality in the horizontal direction implies the bound
\begin{equation}
\label{3.13}
\|u-u_1\|\leq h\|\partial_x(u-u_1)\|\leq h\|\nabla(u-u_1)\|.
\end{equation}
Next, from the harmonicity of $u_1$ we have the obvious identity with a divergence form right-hand side:
$$\Delta (u-u_1)=\Delta u=\frac{\partial((\BM_{\varphi,\psi}^{sym})_{11}-(\BM_{\varphi,\psi}^{sym})_{22})}{\partial x}+\frac{\partial (2\BM_{\varphi,\psi}^{sym})_{12}}{\partial y}
+\partial_y \varphi-\partial_x\psi,$$
thus standard elliptic estimates together with (\ref{3.13}) yield the bound
\begin{equation}
\label{3.14}
\|\nabla(u-u_1)\|\leq \tilde C\left[\|\BM_{\varphi,\psi}^{sym}\|+h(\|\partial_y\varphi\|+\|\partial_x \psi\|)\right].
\end{equation}
Estimates (\ref{3.14}) and (\ref{3.13}) will allow us to replace $u$ by $u_1$ in the targeted estimates, thus proving them under the additional condition that $u$ was harmonic in $R.$
Next we recall lemma~4.3 from [\ref{bib:Gra.Har.1}]. 
\begin{lemma}
\label{lem:3.3}
Suppose $h,p>0$ and $w\in H^1(T,\mathbb R)$ is harmonic and satisfies the boundary condition $w(x,0)=w(x,p)=0$ in the sense of traces, where 
$T=(0,h)\times(0,p).$ Then 
\begin{equation}
\label{3.15}
\|\partial_y w\|_{L^2(T)}^2\leq \frac{2\sqrt 3}{h}\|\partial_x w\|_{L^2(T)}\|w\|_{L^2(T)}+\|\partial_x w\|_{L^2(T)}^2.
\end{equation}
\end{lemma}

We now make use of the fact that $u_1$ is in fact harmonic. To that end we first replace $u$ by $u_1$ through the bounds (\ref{3.13}) and (\ref{3.14}). 
Namely we have by the triangle inequality and by (\ref{3.14}),
\begin{align}
\label{3.16}
\|\partial_y u+\varphi\|^2&\leq 4(\|\partial_y u_1\|^2+\|\partial_y(u-u_1)\|^2+\|\varphi\|^2)\\ \nonumber
&\leq \tilde C(\|\partial_y u_1\|^2+\|\nabla(u-u_1)\|^2+\|\varphi\|^2)\\ \nonumber
&\leq \tilde C(\|\partial_y u_1\|^2+\|\varphi\|^2+\|\BM_{\varphi,\phi}^{sym}\|+h(\|\varphi_y\|+\|\phi_x\|)).
\end{align}
On the other hand we have for the summand $\|\partial_y u_1\|^2$ by Lemma~\ref{lem:3.3},
\begin{align}
\label{3.17}
\|\partial_y u_1\|^2&\leq \frac{2\sqrt 3}{h}\|\partial_x u_1\|\|u_1\|+\|\partial_x u_1\|^2\\ \nonumber
&\leq \frac{2\sqrt 3}{h}(\|\partial_x(u_1-u)\|+\|\partial_x u\|)(\|u\|+\|u_1-u\|)+(\|\partial_x(u_1-u)\|+\|\partial_x u\|)^2\\ \nonumber
&\leq \frac{2\sqrt 3}{h}(\|\nabla(u_1-u)\|+\|\partial_x u\|)(\|u\|+\|u_1-u\|)+(\|\nabla(u_1-u)\|+\|\partial_x u\|)^2.
\end{align}
Recall that that $\partial_x u$ is the $11$ entry of $\BM_{\varphi,\psi}^{sym},$ thus putting together (\ref{3.16}) and (\ref{3.17}), and taking into account the bounds 
(\ref{3.13}) and (\ref{3.14}), we arrive at the estimate 
\begin{align}
\label{3.18}
\|\partial_y u+\varphi\|^2&\leq \tilde C\left(\frac{1}{h}\|u\|\cdot\|\BM_{\varphi,\psi}^{sym}\|+\|u\|^2+\|\BM_{\varphi,\psi}^{sym}\|^2+\|\varphi\|^2\right)\\ \nonumber
&+\tilde C(\|u\|(\|\partial_y \varphi\|+\|\partial_x\psi\|)+h^2(\|\partial_y\varphi\|^2+\|\partial_x\psi\|^2)).
\end{align}
Consequently, for the case $f=\BGa\cdot\BU$ and $g=\BGb\cdot\BU+w$ we have the obvious bounds
\begin{align}
\label{3.19}
\|\partial_y\varphi\|&\leq \tilde C\|\BU\|_{H^1(R)}\leq \tilde C(\|\BM_{\varphi,\psi}^{sym}\|+\|\BU\|+\|w\|),\\ \nonumber
\|\partial_x\psi \|&\leq \tilde C\|\BU\|_{H^1(R)}+\|\partial_xw\|\leq \tilde C(\|\BM_{\varphi,\psi}^{sym}\|+\|\BU\|+\|\partial_xw\|).
\end{align}
Finally, applying the bounds (\ref{3.19}) to the right-hand side of (\ref{3.18}) and estimating the product term $\|u\|(\|\partial_y\varphi\|+\|\partial_x\psi\|)$ 
by $\frac{1}{\delta}\|u\|^2+\delta(\|\partial_y\varphi\|+\|\partial_x\psi\|)^2,$ we obtain the estimate (\ref{3.10}) for $\partial_y u+\varphi$ in place of $\BM$
on the left. The additional $21$ entry $\partial_xv$ of $\BM$ is then estimated in terms of $\partial_y u+\varphi$ and $\BM^{sym}$ via triangle inequality. 
This completes the proof of the lemma.

\end{proof}
Now the bound (\ref{3.8}) is obtained by a clever choice of the fields $\BU,$ $\alpha,$ $\beta$ and the function $w$ in Lemma~\ref{lem:3.2}, which was originally 
done in [\ref{bib:Harutyunyan.3}]. We recall the formulae here leaving the details to the reader. A working choice turns out to be 
$\BU=(u_t,A_zu_z),$ $\BGa=(0,-\Gk_z),$ $\BGb=(A_z^2\Gk_z,-\frac{A_{z,z}}{A_z})$ and the function $w=\frac{A_zA_{z,\Gth}}{A_\Gth}u_\Gth$ in the variables $t$ and $z.$ 
Indeed, we have by straightforward calculation
\begin{align}
\label{3.20}
M_{11}&=u_{t,t},\quad M_{12}=u_{t,z}-A_z\Gk_zu_z\\ \nonumber
M_{21}&=A_zu_{z,t},\quad M_{22}=A_zu_{z,z}+A_z^2\Gk_zu_t+\frac{A_zA_{z,\Gth}}{A_\Gth}u_\Gth,\\ \nonumber
\partial_t w&=\frac{A_zA_{z,\Gth}}{A_\Gth}u_{\Gth,t}\\ \nonumber
\end{align}
thus taking into account the form of the matrix $\BB$ in (\ref{3.4}) we have
\begin{align}
\label{3.21}
M^{sym}_{11}&=B^{sym}_{11},\quad M^{sym}_{12}=A_zB^{sym}_{12},\quad M^{sym}_{22}=A_z^2B^{sym}_{33},\\ \nonumber
|w|&\leq \tilde C|\Bu|,\quad |\partial_t w|\leq \tilde C(|B_{12}|+|\Bu|).
\end{align}
Now (\ref{3.8}) follows from plugging in the forms (\ref{3.20}) and (\ref{3.21}) in (\ref{3.10}) and integrating the obtained estimate in $\Gth\in [0,1]$. 

\end{proof}

For the block $z=const$ we get the analogous inequality (by symmetry) through the same analysis.\\ 

\noindent \textbf{The block $z=const.$} \textit{There exists a constant $C>0,$ such that one has the estimate 
\begin{equation}
 \label{3.22}
\|B_{12}\|^2+\|B_{21}\|^2\leq C\left(\frac{\|u_t\|\|\BB^{sym}\|}{h}+\frac{1}{\delta}\|\Bu\|^2+\delta\|B_{13}\|^2+\|\BB^{sym}\|^2\right).
\end{equation}   
 for all $h\in (0,1),$ all $\Bu\in V^{h,\epsilon},$ and all $\delta\in (0,1)$.}
Finally putting together the estimates (\ref{3.6}), (\ref{3.8}) and (\ref{3.22}) we obtain the bound 
\begin{equation}
 \label{3.23}
\|\BB\|^2\leq C\left(\frac{\|u_t\|\|\BB^{sym}\|}{h}+\frac{1}{\delta}\|\Bu\|^2+\delta\|\BB\|^2+\|\BB^{sym}\|^2\right).
\end{equation}   
Thus choosing the parameter $\delta>0$ small enough (i.e., $C\delta=1/2$) we obtain (\ref{3.5}). This completes the proof of the modified 
estimate (\ref{3.5}). 

\end{proof} 

Although the inequality (\ref{3.5}) will be the one (and not Lemma~\ref{lem:3.1}) frequently used in our analysis, let us mention how Lemma~\ref{lem:3.1} follows from (\ref{3.5}),
which is quite straightforward. Indeed, it is done via the obvious estimates 
\begin{equation}
 \label{3.24}
 \|\nabla\Bu-\BB\|\leq Ch\|\nabla\Bu\|,\quad \|e(\Bu)-\BB^{sym}\|\leq Ch\|\nabla\Bu\|
 \end{equation}
 by Cauchy inequality applied to the product term $\|\BB\|\|u_t\|$ in the form $ab\leq \delta a^2+\frac{1}{\delta} b^2,$ with the parameter $\delta>0$ chosen sufficiently small.

\subsection{Ansatz-free lower bound: Theorem~\ref{2.1}}
\label{sub:3.2} 

We start with the regime $\epsilon\in [h,\sqrt{h}],$ which is the easier one. We have from the Poincar\'e inequality in the $\Gth$ direction that 
$$\|\Bu\|\leq C\epsilon\|\partial_\Gth \Bu\| $$
From the form of the matrix $\BB$ we have 
$$\|\partial_\Gth \Bu \|\leq C(\|\BB\|+\|\Bu\|),$$
thus we have 
$$\|\Bu \|\leq C\epsilon(\|\BB\|+\|\Bu\|)$$
which for sufficiently small $h$ (depending on the constant $C$ right above) implies, taking into account the bound $\epsilon\leq \sqrt{h}$,
\begin{equation}
\label{3.25}
\|\Bu \|\leq C\epsilon\|\BB\|.
\end{equation}
Consequently combining (\ref{3.25}) and (\ref{3.5}) we have for small enough $h$ the bound 
\begin{equation}
\label{3.26}
\|\BB \|\leq \frac{C\epsilon}{h}\|\BB^{sym}\|, 
\end{equation}
regardless of the curvature of $S.$ Now, as in the case of (\ref{3.5}) and Lemma~\ref{3.1} discussed above, the estimate (\ref{2.5}) with $S^{h,\epsilon}$ in place of 
$\Omega^{h,\epsilon}$ follows from (\ref{3.26}) and (\ref{3.24}). As already mentioned below, the passage from $S^{h,\epsilon}$ to $\Omega^{h,\epsilon}$ in the estimates (\ref{2.5})-(\ref{2.8})
will be demonstrated later in Section~3.4.

\subsection{Ansatz-free lower bounds: Theorem~\ref{th:2.2}}
\label{sub:3.3} 

This corresponds to the regime $\epsilon\in (\sqrt{h},1]$ and the curvature matters here as noted in Theorem~\ref{th:2.2}, thus we will consider each case separately. But before getting into each case, we we start by some preliminary calculation following [\ref{bib:Harutyunyan.2}]. The idea is to first eliminate the component $u_t$ 
from the $22$ and $33$ entries of the matrix $\BB,$ where at the same time utilizing the $23$ and $32$ entries of $\BB,$ we end up with an equality involving 
only the in-plane components $u_\Gth,$ $u_z$ and the symmetric part $\BB^{sym},$ which, roughly speaking, will in tern provide us with a Carleman-like 
estimate for the operator $e(\Bu).$ Let $\lambda>0$ be a large parameter to be determined later. Like in Section~\ref{sub:3.1} we will be working 
with the inner product $(f,g)_E.$ We have that  
$$(B^{sym}_{22}, e^{\lambda z}\Gk_zu_z)_E=\int_E e^{\lambda z}A_z\Gk_zu_zu_{\Gth,\Gth}+\int_E e^{\lambda z} A_\Gth A_z\Gk_\Gth\Gk_z u_zu_t+
\int_E e^{\lambda z}A_{\Gth,z} \Gk_z u_z^2,$$
and
$$
(B^{sym}_{33}, e^{\lambda z}\Gk_\Gth u_z)_E=\int_E e^{\lambda z} A_\Gth \Gk_\Gth u_zu_{z,z}+\int_E e^{\lambda z}A_\Gth A_z\Gk_\Gth\Gk_z u_zu_t+
\int_Ee^{\lambda z}A_{z,\Gth} \Gk_\Gth u_\Gth u_z,
$$
thus we can eliminate the $u_t$ term on the right to get
\begin{align}
\label{3.27}
(B^{sym}_{22},& e^{\lambda z}\Gk_zu_z)_E-(B^{sym}_{33}, e^{\lambda z}\Gk_\Gth u_z)_E\\ \nonumber
&=\int_E e^{\lambda z}A_z\Gk_zu_zu_{\Gth,\Gth}+\int_E e^{\lambda z}A_{\Gth,z} \Gk_z u_z^2
-\int_E e^{\lambda z} A_\Gth \Gk_\Gth u_zu_{z,z}-\int_Ee^{\lambda z}A_{z,\Gth} \Gk_\Gth u_\Gth u_z.
\end{align}
The second and the fourth summands on the right of (\ref{3.27}) appear in the desired form, while the first and the third do not. For the third summand we have integrating by parts and utilizing the zero Dirichlet boundary conditions, 
\begin{equation}
\label{3.28}
-\int_E e^{\lambda z} A_\Gth \Gk_\Gth u_zu_{z,z}=\frac{1}{2}\int_\Omega\frac{\partial}{\partial z}(e^{\lambda z} A_\Gth \Gk_\Gth )u_z^2.
\end{equation}
while the adjustment of the first summand is more involved, indeed we have 
\begin{equation}
\label{3.29}
\int_E e^{\lambda z}A_z\Gk_zu_zu_{\Gth,\Gth}=\int_E\frac{\partial}{\partial\Gth}(e^{\lambda z} A_z\Gk_z) u_{\Gth}u_z -\int_E e^{\lambda z} A_z\Gk_zu_{\Gth}u_{z,\Gth}.
\end{equation}
The second summand in (\ref{3.29}) still does not have the required "good" form, thus we modify it as follows. 
Recall the below equality coming from the $23$ and $32$ terms of the matrix $\BB:$ 
$$A_zu_{z,\Gth}=2A_\Gth A_z B_{23}^{sym}+A_{\Gth,z}u_\Gth+A_{z,\Gth}u_z-A_\Gth u_{\Gth,z}$$
hence we can calculate integrating by parts in $\Gth$ as follows:
\begin{align}
\label{3.30}
\int_E &(e^{\lambda z}A_z\Gk_z) u_{\Gth}u_{z,\Gth}\\ \nonumber
&=2\int_E e^{\lambda z} \Gk_zA_\Gth A_zB_{23}^{sym}u_{\Gth}+\int_E e^{\lambda z} \Gk_z A_{\Gth,z}u_{\Gth}^2+
\int_E e^{\lambda z}\Gk_zA_{z,\Gth}u_{\Gth}u_z- \int_E e^{\lambda z} \Gk_z  A_\Gth u_{\Gth}u_{\Gth,z}\\ \nonumber
&=2\int_E e^{\lambda z} \Gk_zA_\Gth A_zB_{23}^{sym}u_{\Gth}+\int_E e^{\lambda z} \Gk_z A_{\Gth,z}u_{\Gth}^2+
\int_E e^{\lambda z}\Gk_zA_{z,\Gth}u_{\Gth}u_z+\frac{1}{2}\int_\Omega \frac{\partial}{\partial z}(e^{\lambda z} \Gk_z A_\Gth)u_{\Gth}^2.
\end{align}
Consequently combining equalities (\ref{3.27})-(\ref{3.30}) we arrive at the key idntity
\begin{align}
\label{3.31}
F(u_\Gth,u_z)=(B_{22}^{sym},e^{\lambda z}\Gk_zu_z)_E-(B_{33}^{sym},e^{\lambda z}\Gk_\Gth u_z)_E+2(B_{23}^{sym}, e^{\lambda z}\Gk_zu_\Gth)_E
\end{align}
where the quadratic form $F(u_\Gth,u_z)$ is given by 
\begin{align}
\label{3.32}
F(u_\Gth,u_z)&=-\int_E \left(\frac{1}{2}\frac{\partial}{\partial z}(e^{\lambda z}\Gk_z A_\Gth)+e^{\lambda z}\Gk_z A_{\Gth,z}\right)u_{\Gth}^2
+\int_E\left(e^{\lambda z}A_{\Gth,z} \Gk_z+\frac{1}{2}\frac{\partial}{\partial z}(e^{\lambda z} A_\Gth \Gk_\Gth )\right)u_z^2\\ \nonumber
&-\int_E\left(e^{\lambda z}A_{z,\Gth}(\Gk_\Gth+\Gk_z)+\frac{\partial}{\partial\Gth}(e^{\lambda z}A_z\Gk_z)\right)u_{\Gth}u_z.
\end{align}
Now having (\ref{3.31}) and (\ref{3.32}) in hand we can proceed with each case separately.\\

\noindent  \textbf{Nagative Gaussian curvature: $K_G<0.$} 
 \begin{proof}
 The purpose of introduction of the function $e^{\lambda z}$ into the analysis and the magic followed by it is as follows. Note that there is a $z$ derivative in both of the coefficients of $u_\Gth^2$ and $u_z^2$ in the quadratic form $F,$ while there is only a $\Gth$ derivative in the coefficient of the product term $u_\Gth u_z.$ 
This means that the coefficients of $u_\Gth^2$ and $u_z^2$ can be made arbitrarily large compared to the coefficient of $u_\Gth u_z,$ by choosing the constant $\lambda$ large enough,
thus making the quadratic form $F$ positive or negative definite provided the coefficients of $u_\Gth^2$ and $u_z^2$ have the same constant sign, which is exactly the case of hyperbolic 
mid-surface $S.$ If $\lambda>0$ is a big enough constant, then according to the bounds (\ref{2.3}) and the condition $K_G<0$ on $E,$ we have that the leading term of the coefficient of 
 $u_\Gth^2$ and $u_z^2$ in $F$ are $\frac{1}{2}\lambda e^{\lambda z}A_\Gth\Gk_\Gth$ and $-\frac{1}{2}\lambda e^{\lambda z}A_\Gth\Gk_z$ respectively, thus there exists constants 
 $C,c>0$ (the constant $c$ depends only on the domain $O(1)$ parameters too) such that we have 
 \begin{equation}
\label{3.33}
|F(u_\Gth,u_z)|\geq C\lambda (\|e^{\frac{\lambda z}{2}}u_\Gth\|^2_{L^2(E)}+\|e^{\frac{\lambda z}{2}}u_z\|^2_{L^2(E)}),\quad \text{for all}\quad \lambda\geq c.
\end{equation} 
 We have on the other hand from (\ref{3.31}) by the Schwartz inequality that 
\begin{equation}
\label{3.34}
|F(u_\Gth,u_z)|\leq C\|e^{\frac{\lambda z}{2}}\BB^{sym}\|_{L^2(E)} (\|e^{\frac{\lambda z}{2}}u_\Gth\|_{L^2(E)}+\|e^{\frac{\lambda z}{2}}u_z\|_{L^2(E)}),\quad \text{for all}\quad \lambda\geq 0.
\end{equation}
Combining now (\ref{3.33}) and (\ref{3.34}) we obtain the Carleman estimate
\begin{equation}
\label{3.35}
\|e^{\frac{\lambda z}{2}}u_\Gth\|_{L^2(E)}+\|e^{\frac{\lambda z}{2}}u_z\|_{L^2(E)}\leq	 \frac{C}{\lambda}\|e^{\frac{\lambda z}{2}}\BB^{sym}\|_{L^2(E)} \quad \text{for all}\quad \lambda\geq c.
\end{equation}
We can now make the choice $\lambda=\frac{c}{\epsilon}$ which clearly satisfies the inequality $\lambda\geq c.$ Also note that one has
the obvious bounds $\frac{cz_1(\Gth)}{\epsilon}\leq \lambda z\leq \frac{cz_1(\Gth)}{\epsilon}+cc_3$ for any $(\Gth,z)\in E,$ thus we get from (\ref{3.35}),
\begin{equation}
\label{3.36}
\|u_\Gth\|_{L^2(E)}+\|u_z\|_{L^2(E)}\leq C\epsilon \|\BB^{sym}\|_{L^2(E)}.
\end{equation}
Consequently, upon squaring and integrating the last inequality in $t\in (-h,h)$ we discover the key inequality 
\begin{equation}
\label{3.37}
\|u_\Gth\|_{L^2(S^h)}^2+\|u_z\|_{L^2(S^{h,\epsilon})}^2\leq C\epsilon^2 \|\BB^{sym}\|_{L^2(S^{h,\epsilon})}.
\end{equation}
We now turn our attention to the estimation of the out-of-plane $u_t$ component. To that end we multiply the equality 
$$B^{sym}_{22}=\Gk_\Gth u_t+\frac{1}{A_\Gth A_z}(A_z u_{\Gth,\Gth}+A_{\Gth,z}u_z).$$
by $A_\Gth A_zu_t$ and integrate the obtained identity over $S^{h,\epsilon},$ where we throw the $\Gth$ derivative of the summand $u_{\Gth,\Gth}$ over $u_t$ by integration by parts. 
Keeping in mind the bound (\ref{3.37}), this leads to the estimate
\begin{equation}
\label{3.38}
\|u_t\|_{L^2(S^{h,\epsilon})}^2\leq C\|\BB^{sym}\|_{L^2(S^{h,\epsilon})}(\|u_t\|_{L^2(S^{h,\epsilon})}+\epsilon\|u_{t,\Gth}\|_{L^2(S^{h,\epsilon})})
\end{equation}
Next we have from the $12$ component of the matrix $\BB,$ and from (\ref{3.36}), that 
\begin{equation}
\label{3.39}
\|u_{t,\Gth}\|_{L^2(S^{h,\epsilon})}\leq C(\|\BB\|_{L^2(S^{h,\epsilon})}+\epsilon \|\BB^{sym}\|_{L^2(S^{h,\epsilon})}),
\end{equation}
hence we get combining (\ref{3.38}) and (\ref{3.39}) the bound 
\begin{equation}
\label{3.40}
\|u_t\|_{L^2(S^{h,\epsilon})}^2\leq C\|\BB^{sym}\|_{L^2(S^{h,\epsilon})}(\epsilon \|\BB\|_{L^2(S^{h,\epsilon})}+\|\BB^{sym}\|_{L^2(S^{h,\epsilon})}).
\end{equation}
In conclusion we combine (\ref{3.40}), (\ref{3.37}) and (\ref{3.5}) to arrive at the desired estimate 
\begin{equation}
\label{3.41}
\|\BB\|_{L^2(S^h)}^2\leq \frac{C\epsilon^{2/3}}{h^{4/3}}\|\BB^{sym}\|_{L^2(S^h)}^2.
\end{equation}
The bound (\ref{2.7}) with $S^{h,\epsilon}$ in place of $\Omega^{h,\epsilon}$ now follows from (\ref{3.41}) via the well-known manner 
through the bounds (\ref{3.24}). The proof for the negative curvature shell case is finished. 
\end{proof}

\noindent \textbf{Positive Gaussian curvature: $K_G>0.$} 

\begin{proof}
The proof of this case consists of several observations. The first observation comes again from the identities (\ref{3.31}) and (\ref{3.32}), where it is clear that by choosing the constant
$\lambda>0$ big enough, we achieve the existence of a constant $C>0$ such that 
\begin{equation}
\label{3.42}
\|u_\Gth\|^2_{L^2(S^h)} \leq Ce^{C\epsilon}(\|u_z\|^2_{L^2(S^h)}+\|\BB^{sym}\|^2_{L^2(S^h)}),\quad \|u_z\|^2_{L^2(S^h)}\leq  Ce^{C\epsilon}(\|u_\Gth\|^2_{L^2(S^h)}+\|\BB^{sym}\|^2_{L^2(S^h)}).
\end{equation}
In the the second step we bound the component $u_t.$ We have for a fixed $t\in (-h,h),$ and the positivity $\Gk_\Gth \Gk_z>0$ that
\begin{equation}
\label{3.43}
(B_{22}^{sym}-\Gk_\Gth u_t, B_{33}^{sym}-\Gk_z u_t)_E \geq C(\|u_t\|^2-\|\BB^{sym}\|^2),
\end{equation}
for some constant $C>0.$ On the other hand recalling the form of the matrix $\BB,$ we have integrating by parts
\begin{align}
\label{3.44}
&(B_{22}^{sym}-\Gk_\Gth u_t, B_{33}^{sym}-\Gk_z u_t)_E\\ \nonumber
&=\int_E \left(u_{\Gth,\Gth}+\frac{A_{\Gth,z}}{A_\Gth}u_z\right)\left(u_{z,z}+\frac{A_{z,\Gth}}{A_z}u_\Gth\right)\\ \nonumber
&=\int_E u_{\Gth,\Gth}u_{z,z}+\int_E\frac{A_{\Gth,z}A_{z,\Gth}}{A_\Gth A_z}u_zu_\Gth
+\int_E\frac{A_{\Gth,z}}{A_\Gth}u_zu_{z,z}+\int_E\frac{A_{z,\Gth}}{A_z}u_\Gth u_{\Gth,\Gth}\\ \nonumber
&=\int_E u_{\Gth,z}u_{z,\Gth}+\int_E\frac{A_{\Gth,z}A_{z,\Gth}}{A_\Gth A_z}u_\Gth u_z-\frac{1}{2}\int_E\frac{\partial}{\partial z}\left(\frac{A_{\Gth,z}}{A_\Gth}\right)u_z^2
-\frac{1}{2}\int_E\frac{\partial}{\partial\Gth}\left(\frac{A_{z,\Gth}}{A_z}\right)u_\Gth^2.
\end{align}
Note that the righthand side of (\ref{3.44}) is a quadratic form in $(u_\Gth,u_z)$ modulo the first summand. Also, that first summand occurs from the following multiplication when 
integrating by parts some of the occurring summands to make them a quadratic:
\begin{align}
\label{3.45}
(B_{23}, &2B^{sym}_{23}-B_{23})_E=\int_E\left(u_{z,\Gth}-\frac{A_{\Gth,z}}{A_\Gth}u_\Gth\right)\left(u_{\Gth,z}-\frac{A_{z,\Gth}}{A_z}u_z\right)\\ \nonumber
&=\int_E u_{\Gth,z}u_{z,\Gth}+\int_E\frac{A_{\Gth,z}A_{z,\Gth}}{A_\Gth A_z}u_zu_\Gth
+\frac{1}{2}\int_E\frac{\partial}{\partial z}\left(\frac{A_{\Gth,z}}{A_\Gth}\right)u_\Gth^2
+\frac{1}{2}\int_E\frac{\partial}{\partial \Gth}\left(\frac{A_{z,\Gth}}{A_z}\right)u_z^2.
\end{align}
Thus putting together (\ref{3.44}) and (\ref{3.45}) we obtain
\begin{equation}
\label{3.46}
(B_{22}^{sym}-\Gk_\Gth u_t, B_{33}^{sym}-\Gk_z u_t)_E=(B_{23}, 2B^{sym}_{23}-B_{23})_E
-\int_E\left(\frac{\partial}{\partial z}\left(\frac{A_{\Gth,z}}{A_\Gth}\right)+\frac{\partial}{\partial \Gth}\left(\frac{A_{z,\Gth}}{A_z}\right)\right)(u_\Gth^2+u_z^2).
\end{equation}
Utilizing the obvious quadratic form estimate $(B_{23}, 2B^{sym}_{23}-B_{23})_E\leq (B^{sym}_{23},B^{sym}_{23})_E,$ and invoking (\ref{3.43}), we discover from (\ref{3.46}) the key 
estimate (after an integration in $t\in (-h,h)$), 
\begin{equation}
\label{3.47}
\|u_t\|^2_{L^2(S^{h,\epsilon})}\leq C(\BB^{sym}_{L^2(S^{h,\epsilon})}+\|u_\Gth\|^2_{L^2(S^{h,\epsilon})}+\|u_z\|^2)_{L^2(S^{h,\epsilon})}.
\end{equation}
Next, Poincar\'e inequality in the $z$ direction implies that 
\begin{equation}
\label{3.48}
\|u_z\|^2_{L^2(S^{h,\epsilon})}\leq C \epsilon^2 \|u_{z,z}\|^2_{L^2(S^{h,\epsilon})}.
\end{equation}
for some constant $C>0.$ Consequently combining the bounds (\ref{3.42}), (\ref{3.47}), and (\ref{3.48}) we arrive at the key inequality
\begin{equation}
\label{3.49}
\|u_z\|^2_{L^2(S^{h,\epsilon})} \leq C\epsilon^2 e^{C\epsilon}(\|\BB^{sym}\|^2_{L^2(S^{h,\epsilon})}+\|u_z\|^2_{L^2(S^{h,\epsilon})}).
\end{equation}
It is clear that if $\epsilon$ is small enough, such that $C\epsilon^2 e^{C\epsilon}<1,$ then (\ref{3.49}) implies the bound $\|u_z\|_{L^2(S^{h,\epsilon})} \leq C\|\BB^{sym}\|_{L^2(S^{h,\epsilon})},$
which in turn would imply the similar bounds $\|u_\Gth\|_{L^2(S^{h,\epsilon})} \leq C\|\BB^{sym}\|_{L^2(S^{h,\epsilon})}$ and $\|u_t\|_{L^2(S^{h,\epsilon})} \leq C\|\BB^{sym}\|_{L^2(S^{h,\epsilon})}$
due to (\ref{3.42}) and (\ref{3.47}). We have thus obtained the following: \textit{There exists $\epsilon_0>0$ depending on the parameters of $S,$ such that one has 
\begin{equation}
\label{3.50}
\|u_t\|_{L^2(S^{h,\epsilon})}+\|u_\Gth\|_{L^2(S^{h,\epsilon})}+\|u_z\|_{L^2(S^{h,\epsilon})} \leq C\|\BB^{sym}\|_{L^2(S^{h,\epsilon})}
\end{equation}
for all $\epsilon\in (0,\epsilon_0).$} Finally we combine (\ref{3.50}) with (\ref{3.5}) to get (\ref{2.6}) with $S^{h,\epsilon}$ in place of $\Omega^{h,\epsilon}.$

\end{proof}

\noindent \textbf{Zero Gaussian curvature: $\Gk_z=0, |\Gk_\Gth|\geq K>0.$} 

\begin{proof}
We using the Gauss-Codazzi relations [\ref{bib:Tov.Smi.}, Section~1.1] one can obtain a rather explicit form for the metric tensor components $A_\Gth$ and $A_z$ as well as for the nonzero curvature $\Gk_\Gth.$ 
Indeed, taking into account the equality $\Gk_z=0,$ the Gauss-Codazzi relations 
\begin{align}
  \label{3.51}
  &\dif{\Gk_{z}}{\Gth}=(\Gk_{\Gth}-\Gk_{z})\frac{A_{ z,\Gth}}{A_{ z}},\qquad
\dif{\Gk_{\Gth}}{ z}=(\Gk_{z}-\Gk_{\Gth})\frac{A_{\Gth, z}}{A_{\Gth}},\\ \nonumber
& \dif{}{ z}\left(\frac{A_{\Gth, z}}{A_{ z}}\right)+
\dif{}{\Gth}\left(\frac{A_{ z,\Gth}}{A_{\Gth}}\right)=-A_{ z}A_{\Gth}\Gk_{z}\Gk_{\Gth},
\end{align}
reduce to 
\begin{equation*}
A_{z,\Gth}=0, \quad \frac{\partial \Gk_\Gth}{\partial z}=-\Gk_\Gth\frac{A_{\Gth,z}}{A_\Gth},\quad \frac{\partial }{\partial z} \left(\frac{A_{\Gth,z}}{A_z}\right)=0,
\end{equation*}
solving which we get the explicit forms 
\begin{equation}
  \label{3.52}
  A_{z}=B'(z),\qquad A_{\Gth}=a(\Gth)B(z)+b(\Gth),\qquad\Gk_{\Gth}=\frac{c(\Gth)}{A_{\Gth}},
\end{equation}
where the functions $A_{z}$, $A_{\Gth}$ and $c(\Gth)$ are strictly positive on $E.$ Thus we get the following forms for the gradient and the matrix $\BB:$ 
\begin{equation}
  \label{3.53}
\nabla\Bu=
\begin{bmatrix}
  u_{t,t} & \dfrac{u_{t,\Gth}-c(\Gth)u_{\Gth}}{A_{\Gth}+tc(\Gth)} &
\dfrac{u_{t,z}}{A_{z}}\\[3ex]
u_{\Gth,t}  &
\dfrac{u_{\Gth,\Gth}+c(\Gth)u_{t}+a(\Gth)u_{z}}{A_{\Gth}+tc(\Gth)} &
\dfrac{u_{\Gth,z}}{A_{z}}\\[3ex]
u_{ z,t}  & \dfrac{u_{z,\Gth}-a(\Gth)u_{\Gth}}{A_{\Gth}+tc(\Gth)} &
\dfrac{u_{z,z}}{A_{z}}
\end{bmatrix},
\quad
 \BB=
\begin{bmatrix}
  u_{t,t} & \dfrac{u_{t,\Gth}-c(\Gth)u_{\Gth}}{A_{\Gth}} &
\dfrac{u_{t,z}}{A_{z}}\\[3ex]
u_{\Gth,t}  &
\dfrac{u_{\Gth,\Gth}+c(\Gth)u_{t}+a(\Gth)u_{z}}{A_{\Gth}} &
\dfrac{u_{\Gth,z}}{A_{z}}\\[3ex]
u_{ z,t}  & \dfrac{u_{z,\Gth}-a(\Gth)u_{\Gth}}{A_{\Gth}} &
\dfrac{u_{z,z}}{A_{z}}
\end{bmatrix}.
\end{equation}
Note first that by Poincar\'e inequality we can bound in this case,
\begin{equation}
  \label{3.54}
\|u_z\|^2_{L^2(S^{h,\epsilon})} \leq \epsilon^2 \|u_{z,z}\|^2_{L^2(S^{h,\epsilon})}\leq C\epsilon^2 \|\BB^{sym}\|^2_{L^2(S^{h,\epsilon}))},
\end{equation} 
as $u_{z,z}=A_zB_{33}^{sym}.$ We have on the other hand integrating by parts
\begin{align}
\label{3.55}
(B_{23}, B_{32})_E&=\int_E u_{\Gth,z}(u_{z,\Gth}-a(\Gth)u_\Gth)\\ \nonumber 
&=\int_E u_{\Gth,\Gth}u_{z,z}.
\end{align}
We have further 
\begin{align}
\label{3.56}
(B_{22}^{sym}, B_{33}^{sym})_E&=\int_E u_{z,z}(u_{\Gth,\Gth}+c(\Gth)u_t+a(\Gth)u_z),
\end{align}
 thus combining (\ref{3.55}) and (\ref{3.56}) and taking into account the equality $B_{32}=2B_{23}^{sym}-B_{23},$ we obtain
$$(B_{23}, 2B_{23}^{sym}-B_{23})_E=(B_{33}^{sym},B_{22}^{sym}-c(\Gth)u_t-a(\Gth)u_z),$$
which yields the estimate (upon integration in $t\in (-h,h)$),
\begin{equation}
  \label{3.57}
\|B_{23}\|_{L^2(S^{h,\epsilon})}\leq C \|\BB^{sym}\|_{L^2(S^{h,\epsilon})}(\|\BB^{sym}\|_{L^2(S^{h,\epsilon})}+\|u_t\|_{L^2(S^{h,\epsilon})}+\|u_z\|_{L^2(S^{h,\epsilon})}).  
 \end{equation} 
 Note next that an application of (\ref{3.54}) simplifies (\ref{3.57}) to the form 
\begin{equation}
  \label{3.58}
\|B_{23}\|_{L^2(S^{h,\epsilon})}^2\leq C \|\BB^{sym}\|_{L^2(S^{h,\epsilon})}(\|\BB^{sym}\|_{L^2(S^{h,\epsilon})}+\|u_t\|_{L^2(S^{h,\epsilon})}).  
 \end{equation} 
Next we estimate the component $u_\Gth$ by noting that Poincar\'e inequality gives on one hand the bound 
$$\|u_\Gth\|^2_{L^2(S^{h,\epsilon})} \leq \epsilon^2 \|u_{\Gth,z}\|^2_{L^2(S^{h,\epsilon})},$$  
 and on the other hand we have $u_{\Gth,z}=A_zB_{23}$, thus taking into account the estimate (\ref{3.58}) we get 
 \begin{equation}
  \label{3.59}
\|u_\Gth\|_{L^2(S^{h,\epsilon})}^2\leq C\epsilon^2 \|\BB^{sym}\|_{L^2(S^{h,\epsilon})}(\|\BB^{sym}\|_{L^2(S^{h,\epsilon})}+\|u_t\|_{L^2(S^{h,\epsilon})}).  
 \end{equation}  
In order to estimate the out-of-plane $u_t$ component we calculate the inner product by integration by parts:
\begin{align}
\label{3.60}
\left(\frac{u_t}{A_z},B_{22}^{sym}\right)_E&=\int_E u_t(u_{\Gth,\Gth}+c(\Gth)u_t+a(\Gth)u_z)\\ \nonumber
&=\int_E u_{t,\Gth}u_{\Gth}+\int_E c(\Gth)|u_t|^2+\int_E a(\Gth)u_zu_t.
\end{align}
Integrating (\ref{3.60}) in $t\in (-h,h)$ and applying the Schwartz inequality to the product terms we get 
\begin{equation}
\label{3.61}
\|u_t\|_{L^2(S^{h,\epsilon})}^2\leq C\|u_t\|_{L^2(S^{h,\epsilon})}(\|\BB^{sym}\|_{L^2(S^{h,\epsilon})}+\|u_z\|_{L^2(S^{h,\epsilon})})+\|u_\Gth\|_{L^2(S^{h,\epsilon})}\|u_{t,\Gth}\|_{L^2(S^{h,\epsilon})}.
\end{equation}   
 Next we have $u_{t,\Gth}=A_\Gth B_{12}+c(\Gth)u_\Gth,$ thus we have 
 \begin{equation}
\label{3.62}
\|u_{t,\Gth}\|_{L^2(S^{h,\epsilon})} \leq C(\|\BB\|_{L^2(S^{h,\epsilon})}+\|u_{\Gth}\|_{L^2(S^{h,\epsilon})}).
\end{equation} 
It is now easy to see that putting together the estimates (\ref{3.5}), (\ref{3.54}), (\ref{3.59}), (\ref{3.61}), and  (\ref{3.62}) we discover 
\begin{equation}
\label{3.63}
\|u_t\|_{L^2(S^{h,\epsilon})} \leq \frac{C\epsilon}{\sqrt{h}}\|\BB^{sym}\|_{L^2(S^{h,\epsilon})}.
\end{equation} 
finally we combine (\ref{3.5}), (\ref{3.54}), (\ref{3.59}), and (\ref{3.63}) to get the estimate 
\begin{equation}
\label{3.64}
\|\BB\|_{L^2(S^{h,\epsilon})}^2 \leq \frac{C\epsilon}{h\sqrt{h}}\|\BB^{sym}\|_{L^2(S^{h,\epsilon})}^2.
\end{equation} 
As already understood before, the inequality (\ref{2.8}) (with $S^{h,\epsilon}$ in place of $\Omega^{h,\epsilon}$) follows from the similar estimate 
(\ref{3.64}) through the bounds (\ref{3.24}). 
\end{proof}

\subsection{Passage from $S^{h,\epsilon}$ to $\Omega^{h,\epsilon}$ in Theorems~2.1 and 2.2}

The passage from the shell $S^{h,\epsilon}$ to the thin domain $\Omega^{h,\epsilon}$ is done in a well-known fashion through a localization argument. 
The idea is based on the fact that if the domain $\Omega$ has comparable dimensions in three mutually orthogonal directions, then the constant $C_1$ in (\ref{1.3}) 
depends only on the Lipschitz constant of $\Omega$ and the constants controlling the size ratio in that orthogonal directions. We recall the following simple 
lemma proven in [\ref{bib:Harutyunyan.3}, Lemma~5.2]. 

\begin{lemma}
\label{lem:3.4}
Assume $D_1\subset D_2\subset \mathbb R^3$ are open bounded connected Lipschitz domains. By Korn's first inequality, there exist constants $K_1$ and $K_2$ such that for any vector field $\BU\in W^{1,2}(D_2,\mathbb R^3),$ there exist skew-symmetric matrices $\BA_1,\BA_2\in\mathbb M^{3\times 3},$ such that 
\begin{equation}
\label{3.65}
\|\nabla\BU-\BA_1\|_{L^2(D_1)} \leq K_1\|e(\BU)\|_{L^2(D_1)},\quad \|\nabla\BU-\BA_2\|_{L^2(D_2)} \leq K_2\|e(\BU)\|_{L^2(D_2)}.
\end{equation}
The assertion is that there exists a constant $C>0$ depending only on the quantities $K_1,K_2$ and $\frac{|D_2|}{|D_1|},$ such that for any 
vector field $\BU\in W^{1,2}(D_2, \mathbb R^n)$ one has
\begin{equation}
\label{3.66}
\|\nabla\BU\|_{L^2(D_2)} \leq C(\|\nabla\BU\|_{L^2(D_1)}+\|e(\BU)\|_{L^2(D_2)}).
\end{equation}
\end{lemma}

In order to utilize the lemma, we divide $\Omega$ into small parts with size of order $h$ and extend the existing local estimate on all smaller parts in the normal to $S$ direction to the bigger (but still of order $h$) parts containing it. Let $\Bu=(\bar u_1,\bar u_2, \bar u_3)$ be $\Bu$ in Cartesian coordinates $\Bx=(x_1,x_2,x_3),$ and let
$D\Bu$ denote the Cartesian gradient. We divide the domains $\Omega^{h,\epsilon}$ into small pieces of order $h$ by dividing each parameter range $\Gth$ and $z$ 
into small intervals of order $h.$ We will then get roughly $\frac{1}{h}\cdot\frac{\epsilon}{h}=\frac{\epsilon}{h^2}$ small domains of order $h$ with comparable size in the directions 
$t,\Gth,z$ and with uniform Lipschitz constants. Denote the small domains by $\Omega_i^{h,\epsilon},$ where $i=1,2,\dots,N,$ where $N=O(\frac{\epsilon}{h^2})$ and set 
$S_i^{h,\epsilon}=S^{h,\epsilon}\cap\Omega_i^{h,\epsilon}.$ We have by the lemma
\begin{equation}
\label{3.67}
\|D\bar\Bu\|_{L^2(\Omega_i^{h,\epsilon})} \leq C(\|D\bar\Bu\|_{L^2(S_i^{h,\epsilon})}+\|e(\bar\Bu)\|_{L^2(\Omega_i^{h,\epsilon})}),\quad   i=1,2,\dots,N,
\end{equation} 
with a uniform constant $C$ depending only on the thin domain $O(1)$ parameters. Consequently, summing the estimates (\ref{3.67}) in $i=1,2,\dots,N$ we discover
\begin{equation}
\label{3.68}
\|D\bar\Bu\|_{L^2(\Omega^{h,\epsilon})} \leq C(\|D\bar\Bu\|_{L^2(S^{h,\epsilon})}+\|e(\bar\Bu)\|_{L^2(\Omega^{h,\epsilon})}).
\end{equation} 
Finally, combining (\ref{3.68}) with the versions of (\ref{2.5})-(\ref{2.8}) with $S^{h,\epsilon}$ in place of $\Omega^{h,\epsilon},$ we obtain the original versions of 
(\ref{2.5})-(\ref{2.8}).

\section{Ans\"atze realizing the upper bounds}
\setcounter{equation}{0}
\label{sec:4}

\subsection{The regime $\epsilon \in [h, \sqrt{h}]$}
\label{sub:4.1} 

In this regime the construction of the Ansatz is based on an observation made in [\ref{bib:Gra.Har.4},\ref{bib:Harutyunyan.2}]. Namely, one assumes first that 
the displacement $\Bu$ is smooth and depends on the normal variable $t$ linearly, i.e., $\Bu(t,\Gth,z)=\Bu_1(\Gth,z)+t\Bu_2(\Gth,z).$ This leads to the equality 
$$e(\Bu)=\BA_1(\Bu_1,\Bu_2)+t\BA_2(\Bu_1,\Bu_2)$$
 for the symmetrized gradient, where $\BA_1$ and $\BA_2$ are $3\times 3$ matrices depending on $\Bu_1$ and $\Bu_2.$ The next hypothesis is to choose the displacements
$\Bu_1$ and $\Bu_2$ such that $11$, $12$ and $13$ components of the summand $\BA1(\Bu_1,\Bu_2)$ vanish. An easy analysis of the occurring PDEs leads to the following form 
\begin{equation}
\label{4.1}
\begin{cases}
u_t=w,\\
u_\Gth=v-t\left(\frac{w_{,\Gth}}{A_\Gth}-\kappa_\Gth v\right),\\
u_z=s-t\left(\frac{w_{,z}}{A_z}-\kappa_z s\right),
\end{cases}
\end{equation}
where $w,v$ and $s$ are smooth functions in $\Gth$ and $z.$ Next we choose $v=s=0$ and the function $w$ having an oscillation in $\Gth$ as follows:
\begin{equation} 
\label{4.2}
u_t = W\left( \tfrac{\theta}{\sqrt{h}}, z\right),
\quad
u_\theta = - \frac{t}{A_\theta \sqrt{h}} W_{,\theta}\left( \tfrac{\theta}{\sqrt{h}}, z\right),
\quad
u_z = - \frac{t}{A_z} W_{,z}\left( \tfrac{\theta}{\sqrt{h}}, z\right),
\end{equation}
where $W = W(\theta, z)\colon \mathbb R^2\to\mathbb R$ is a smooth function compactly supported on the set $E$ and satisfying the conditions:
\begin{equation} 
\label{4.3}
|W| \sim 1, \quad |W_{,z}| \sim \epsilon^{-1} \quad \text{and} \quad |W_{,zz}| \sim \epsilon^{-2}.
\end{equation}
Moreover, the above estimates also hold if we replace $W$ by its first and second order $\theta$-derivatives. We get by straightforward calculations that
\begin{equation*}
e(\bm{u})_{11} = e(\bm{u})_{12} = e(\bm{u})_{13} = 0, 
\end{equation*}
moreover
\begin{equation*}
\begin{split}
e(\bm{u})_{22} &= \frac{A_z^2 A_\theta^3 \kappa_\theta W - tA_\theta^2 A_{\theta,z} W_{,z} + t h^{-\frac{1}{2}} A_z^2 A_{\theta, \theta} W_{,\theta} - t h^{-1} A_z^2 A_\theta W_{,\theta \theta}}{ A_z^2 A_\theta^3 (1+t \kappa_\theta)},
\\[.1in]
e(\bm{u})_{23} &= \frac{t h^{-\frac{1}{2}} (2 + t \kappa_\theta + t \kappa_z)}{2 A_z^2 A_\theta^2 (1+t \kappa_\theta) (1+t \kappa_z)} \left[ \sqrt{h} A_{z,\theta} A_\theta W_{,z} + A_z \left( A_{\theta,z} W_{,\theta} - A_\theta W_{,\theta z} \right) \right],
\\[.1in]
e(\bm{u})_{33} &= \frac{A_\theta^2 \left[ A_z^3 \kappa_z W + t A_{z,z} W_{,z} \right] - t h^{-\frac{1}{2}} A_z^2 A_{z,\theta} W_{,\theta} - t A_z A_\theta^2 W_{,zz}}{A_z^3 A_\theta^2 (1+t \kappa_z)}.
\end{split}
\end{equation*}
Therefore, we have that
\begin{equation} 
\label{4.4}
\|e(\bm{u})\| \leq C\cdot\max \left\{ (h\epsilon)^{1/2}, \left(\frac{h}{\epsilon}\right)^{3/2}, \frac{h}{\epsilon^{1/2}} \right\}.
\end{equation}
Next, it is easy to compute
\begin{equation*}
(\nabla \bm{u})_{12} = \frac{W_{,\theta}}{\sqrt{h} A_\theta}, 
\qquad \qquad
(\nabla \bm{u})_{13} = \frac{W_{,z}}{A_z},
\end{equation*}
and hence
\begin{equation} 
\label{4.5}
\|\nabla \bm{u}\| \geq C\cdot \max \left\{ \left(\frac{h}{\epsilon}\right)^{1/2},\epsilon^{1/2} \right\}.
\end{equation}
Thus, in the regime $\epsilon \in [h, \sqrt{h}]$ from (\ref{4.4}) and (\ref{4.5}) we find that $\|e(\bm{u})\| \leq C  \left(\frac{h}{\epsilon}\right)^{3/2}$ 
and $\|\nabla \bm{u}\| \geq C \left(\frac{h}{\epsilon}\right)^{1/2}$ so that
\begin{equation} 
\label{4.6}
\|e(\bm{u})\|^2 \leq \frac{Ch^2}{\epsilon^2} \|\nabla \bm{u}\|^2, 
\end{equation}
which holds true no matter whether the Gaussian curvature of the shell is positive, negative or zero.

\subsection{The regime $\epsilon \in [\sqrt{h}, 1]$}
\label{sub:4.2}

\textbf{Positive Gaussian curvature: $K_G>0.$} It is easy to see that in this case too the Ansatz (\ref{4.2}) attains the optimal bound. Indeed, from (\ref{4.4}) and (\ref{4.5}) we find that 
$\|e(\bm{u})\| \leq C(h\epsilon)^{1/2}$ and $\|\nabla \bm{u}\| \geq C\epsilon^{1/2},$ so that
\begin{equation*}
\|e(\bm{u})\|^2 \leq C h \|\nabla \bm{u}\|^2.
\end{equation*}

\noindent \textbf{Negative Gaussian curvature: $K_G<0.$} In this case the Ansatz comes from a modification of an Ansatz constructed by 
Tovstik and Smirnov in [\ref{bib:Tov.Smi.}] for negative curvature shells in the regime $\epsilon\sim 1.$ The first construction steps are the same, i.e., 
let $\Bu$ be above, satisfying (\ref{4.1}) and let $W\colon \mathbb R^2\to\mathbb R$ be a smooth function compactly supported on $E$ satisfying (\ref{4.3}) and 
the analogous conditions for its first and second $\Gth$ derivatives. 
For $n = (\epsilon h)^{- \frac{1}{3}}$ make the choice
\begin{equation}
\label{4.7}
\begin{cases}
w = n W(\theta, z) \sin \left( n f(\theta, z) \right)
\\[.1in]
v = A_\theta \kappa_\theta \frac{W(\theta, z)}{f_{,\theta}(\theta, z)} \cos \left( n f(\theta, z) \right)
\\[.1in]
s = A_z \kappa_z \frac{W(\theta, z)}{f_{,z}(\theta, z)} \cos \left( n f(\theta, z) \right),
\end{cases}
\end{equation}
where $f$ solves the transport equation
\begin{equation} 
\label{4.8}
\frac{\kappa_\theta}{A_z^2} f_{,z}^2 + \frac{\kappa_z}{A_\theta^2} f_{,\theta}^2 = 0.
\end{equation}
Again all the entries in the first row of $e(\bm{u})$ are zero and one can show by direct calculation that
\begin{equation}
\label{4.9}
\|e(\bm{u})_{22}\| \leq  C\left(\frac{h}{\epsilon}\right)^{1/2},
\quad\text{and}\quad \|e(\bm{u})_{33}\| \leq C\left(\frac{h}{\epsilon}\right)^{1/2}. 
\end{equation}
We have further,
\begin{equation} 
\label{4.10}
e(\bm{u})_{23} = - n \frac{ \left[ f_{,z}^2 A_\theta^2 \kappa_\theta + f_{,\theta}^2 A_z^2 \kappa_z \right] \sin(n f) W}{2A_z A_\theta f_{,z} f_{,\theta} (1+t \kappa_z) (1+t \kappa_\theta)} 
+ \frac{1}{\epsilon}\cdot\frac{f_{,z} f_{,\theta} \sin(nf) W}{A_z A_\theta (1+t \kappa_z) (1+t \kappa_\theta)} + \frac{o(1)}{\epsilon}
\end{equation}
where $o(1)\to 0$ as $h\to 0.$ The fact that $f$ solves the PDE in (\ref{4.8}) becomes crucial at this stage to conclude that the first term in (\ref{4.10}) vanishes and therefore 
$\|e(\bm{u})_{23}\| \leq \frac{Ch^{1/2}}{\epsilon}$. Thus, we conclude
\begin{equation} 
\label{4.11}
\|e(\bm{u})\| \leq C\left(\frac{h}{\epsilon}\right)^{1/2}.
\end{equation} 
Next we compute
\begin{equation*}
[\nabla \bm{u}]_{12} = \frac{W \cos(nf) \left[ n^2 f_{,\theta}^2 - \kappa_\theta^2 A_\theta^2 \right] + n W_{,\theta} f_{,\theta} \sin(nf)}{A_\theta f_{,\theta}},
\end{equation*}
and hence
\begin{equation*}
\|\nabla \bm{u}\| \geq C (h\epsilon)^{1/2} n^2 = \frac{C}{(h\epsilon)^{1/6}}.
\end{equation*}
Combining the latter inequality with (\ref{4.11}) we obtain the inequality
\begin{equation*}
\|e(\bm{u})\|^2 \leq  \frac{Ch^{4/3}}{\epsilon^{2/3}} \|\nabla \bm{u}\|^2.
\end{equation*}

\noindent \textbf{Zero Gaussian curvature: $\Gk_z=0, \Gk_\Gth>0.$} The construction in this case is based on the similar Ansatz constructed in [\ref{bib:Gra.Har.4}], 
which we modify for any $\epsilon\in [\sqrt{h},1]$ here. Consider the vector field $\bm{u} = \bm{v}(\theta, z) + t \bm{w} (\theta, z)$, where 
$\bm{w}=(w_t,w_\Gth,w_z)$ and $\bm{v}=(v_t,v_\Gth,v_z)$ are expressed in terms of the last two components of $\bm{v}$ according to the following formulas

\begin{equation*}
w_t = 0, \qquad
w_\theta = \frac{c(\theta) v_\theta - v_{t,\theta}}{A_\theta}, \qquad
w_z = - \frac{v_{t,z}}{A_z},
\qquad
v_t = - \frac{v_{\theta, \theta} + a(\theta) v_z}{c(\theta)},
\end{equation*}
and $v_\theta, v_z$ solve the PDE
\begin{equation} 
\label{4.12}
-A_\theta v_{\theta,z} = A_z (v_{z,\theta} - a(\theta) v_\theta),
\end{equation}
the solvability of which is analyzed below.\\
\noindent \textbf{Case 1:} \textit{Assume 
$$\frac{A_z}{A_\theta} = \frac{H(\theta)}{G(z)}$$
for some $C^1$ functions $H$ and $G$.}\\
It is easy to see that this condition is equivalent to the functions $a(\Gth)$ and $b(\Gth)$ in (\ref{3.52}) being linearly dependent, i.e. there exists a constant scalar $\lambda_0$ such that $b(\theta) = \lambda_0 a(\theta)$ (or $a(\theta) = \lambda_0 b(\theta)$ which can be analyzed analogously). Under this assumption, the PDE in 
(\ref{4.12}) has a general solution
\begin{equation*}
v_z = A_\theta G H \phi_{,z}, \qquad v_\theta = - A_\theta H^2 \phi_{,\theta},
\end{equation*} 
where $\phi$ is an arbitrary smooth function compactly supported on $E.$ Let us now make the choice
$$\phi(\theta, z) = \Phi(n \theta, z) \qquad \text{with} \qquad n = \epsilon^{-\frac{1}{2}} h^{-\frac{1}{4}},$$

\noindent where $\Phi(\theta, z)$ is a smooth, $1$-periodic function in $\theta$ that satisfies the same estimates as $W$ in (\ref{4.3}), moreover the estimates also hold for up to the third order
$\theta$-derivatives of $\Phi.$ The entries in the first row of $e(\bm{u})$ are zero and direct calculations lead to the scaling estimate
\begin{equation} 
\label{4.13}
\|e(\bm{u})\| \leq \frac{Ch^{1/2}}{\epsilon^{3/2}}.
\end{equation}
Next we can compute
\begin{equation*}
[\nabla \bm{u}]_{12} = \frac{\Phi_{,\theta \theta \theta}}{a^2 c \epsilon^{\frac{3}{2}} h^{\frac{3}{4}}} + \frac{o(1)}{\epsilon^{\frac{3}{2}} h^{\frac{3}{4}}},
\end{equation*}
and therefore
\begin{equation} 
\label{4.14}
\|\nabla \bm{u} \| \geq \frac{C(h\epsilon)^{1/2}}{\epsilon^{\frac{3}{2}} h^{\frac{3}{4}}}.
\end{equation}
Putting together (\ref{4.13}) and (\ref{4.14}) we conclude
\begin{equation}
 \label{4.15}
\|e(\bm{u})\|^2 \leq \frac{Ch^{\frac{3}{2}}}{\epsilon} \|\nabla \bm{u} \|^2.
\end{equation}

\noindent \textbf{Case 2:} \textit{Assume there exists an interval $I \subset (0,p)$, such that $a \neq 0$ and $\rho'\neq 0$ on $I$, where $\rho(\theta) = b(\theta) / a(\theta)$.}\\
In this case we find
\begin{equation*}
v_\theta = \frac{1}{a(\theta)} \frac{\partial}{\partial \theta} \left( \frac{\phi_{,\theta}}{\rho'(\theta)} \right),
\qquad
v_z = \frac{B'(z) \phi_{,\theta} + \rho'(\theta) \phi_{,z} - (B(z) + \rho(\theta)) \phi_{,\theta z} }{B'(z) \rho'(\theta)},
\end{equation*}

\noindent where we take

$$\phi(\theta, z) = \eta(\theta, z) \Phi(n \theta, z)$$

\noindent with $\eta$ a smooth $p$-periodic function supported on the set $E_1=\{(\Gth,z) \ : \ \Gth\in I, \ z_1(\Gth)<z<z_2(\Gth)\},$ and $\Phi$ and $n$ are as in Case 1. In this case too 
one can show that the estimates (\ref{4.13}) and (\ref{4.14}) hold true, which then imply (\ref{4.15}).

\section{Rigidity without boundary conditions}
\setcounter{equation}{0}
\label{sec:5}

In this section we give a briefly discussion on how the rigidity changes if one removes the zero Dirichlet boundary conditions on the displacement $\Bu$ on the 
thin part of the boundary of $\partial\Omega^{h,\epsilon}.$ This is particularly transparent in the regime $\epsilon\in [h,\sqrt{h}]$ that we discuss below. Moreover, it is worth mentioning that 
we believe, based on the works [\ref{bib:Fri.Jam.Mue.1},\ref{bib:Fri.Jam.Mue.2},\ref{bib:Fri.Jam.Mor.Mue.},\ref{bib:Hor.Lew.Pak.},\ref{bib:Gra.Har.4},\ref{bib:Harutyunyan.2}], that 
rigidity of shells (thin domains) decreases in the regime $\epsilon<<1$ once one removes the boundary conditions, this is task for future analysis.\\ 

\noindent \textbf{The regime $\epsilon\in [h,\sqrt{h}].$} Assume $S$ has a constant sign Gaussian curvature, i.e., either $K>0$ or $K<0$ on $S^{h,\epsilon}.$ 
Fix a point $p=(\Gth_0,z_0)\in E,$ and consider the principal curve $\Gamma: \gamma(\Gth)=\BR(\Gth,z_0)$ ($\Gth\in [0,1]$) passing through $p.$ 
For each point $\gamma(\Gth)\in S^{h,\epsilon},$ denote the $\pi_{\gamma(\Gth)}$ the tangent plane to $S^{h,\epsilon}$ at $\gamma(\Gth).$ Because $K$ has a constant sign, the surface 
$S^{h,\epsilon}$ will find itself always on the same side of $\pi_{\gamma(\Gth)}$ as the point $\gamma(\Gth)$ moves along the curve $\Gamma.$
We construct a new developable surface $\tilde S^{h,\epsilon}$ as follows: For each point $\gamma(\Gth),$ let $\BGt(\Gth)\subset \pi_{\gamma(\Gth)}$ 
be the normal to the curve $\GG$ (note that $\Bn(\Gth,z_0)$ will be the binormal to $\Gamma$). Next we project the line $l(\Gth)$ passing through the point $p$ and 
having the direction of $\BGt(\Gth)$ onto the $(\Gth,z)$ plane and denote the intersection of the projection with the parametrization set $E$ by $\alpha(\Gth).$ 
Finally we denote the part of $l(\Gth)$ that projects onto the segment $\alpha(\Gth)$ by $\beta(\Gth)$ and consider the union 
$$\tilde S^{h,\epsilon}=\cup_{\Gth\in [0,1]}\beta(\Gth).$$ 
From the definition the surface $\tilde S^{h,\epsilon}$ apparently is as regular as $S^{h,\epsilon}$ and it has zero Gaussian curvature being a union of straight segments. On the 
other hand as the principal curvatures of $S^{h,\epsilon}$ are bounded in the absolute value by the constant $k<\infty$ (see (\ref{2.3})), then the deviation of $S^{h,\epsilon}$ from 
$\tilde S^{h,\epsilon}$ at every point $\Br(\Gth,z)\in S$ will be bounded by $Ck\epsilon^2\leq Ckh$ bearing in mind that $\epsilon\in [h,\sqrt{h}].$ This being said, 
the thin domain $\Omega^{h,\epsilon}$ can be embedded in the shell around $\tilde S^{h,\epsilon}$ with thickness $Ch$ for some fixed $C>0$ for small enough $h>0.$ 
But the rigidity of developable shells with thickness of order $h$ is $h^{-2}$ [\ref{bib:Hor.Lew.Pak.}], i.e, both constants $C$ and $C_1$ in (\ref{1.2}) and (\ref{1.3}) scale like $ch^{-2},$ i.e.,
the result is independent of $\epsilon$, which is in contrast to (\ref{2.7})-(\ref{2.8}). Thus the rigidity dropped upon removing the boundary conditions on $\Bu.$
Note that the same argument works also for developable shells, i.e., in the case of (\ref{2.8}).\\
 
\noindent \textbf{The regime $\epsilon\in [h,\sqrt{h}].$} In this regime the picture is rather clear in the situation of (\ref{2.8}), i.e., when one has $\Gk_z=0$ and 
$|\Gk_\Gth|>0$ on $E.$ In this case we again have that the rigidity of $\Omega^{h,\epsilon}$ without boundary conditions on $\Bu$ is $h^{-2}$ [\ref{bib:Hor.Lew.Pak.}],
while with boundary conditions it is $\epsilon h^{-3/2},$ which is evidently much bigger.

\section*{Acknowledgements}
D. H. would like to thank National Science Foundation for support under Grants No. DMS-1814361.

\end{document}

%% file: mymacros.tex
\usepackage{amssymb,bm}
\usepackage{amsmath}
\usepackage{amsthm}
\usepackage{graphicx}
\usepackage[active]{srcltx} 
\usepackage{hyperref}
\hypersetup{pdfborder=0 0 0}

\newtheorem{theorem}{{\sc Theorem}}[section]

\newtheorem{lemma}[theorem]{{\sc Lemma}}

\newtheorem{remark}[theorem]{Remark}

\newcommand{\dif}[2]{\displaystyle\frac{\partial #1}{\partial #2}}

\def\XXint#1#2#3{{\setbox0=\hbox{$#1{#2#3}{\int}$ }
\vcenter{\hbox{$#2#3$ }}\kern-.6\wd0}}

\newcommand{\Gk}{\kappa}

\newcommand{\Gth}{\theta}

\newcommand{\GG}{\Gamma}

\bmdefine\BGa{\alpha}
\bmdefine\BGb{\beta}
\bmdefine\BGd{\delta}
\bmdefine\BGe{\epsilon}
\bmdefine\BGve{\varepsilon}
\bmdefine\BGf{\phi}
\bmdefine\BGvf{\varphi}
\bmdefine\BGg{\gamma}
\bmdefine\BGc{\chi}
\bmdefine\BGi{\iota}
\bmdefine\BGk{\kappa}
\bmdefine\BGl{\lambda}
\bmdefine\BGn{\eta}
\bmdefine\BGm{\mu}
\bmdefine\BGv{\nu}
\bmdefine\BGp{\pi}
\bmdefine\BGth{\theta}
\bmdefine\BGvth{\vartheta}
\bmdefine\BGr{\rho}
\bmdefine\BGvr{\varrho}
\bmdefine\BGs{\sigma}
\bmdefine\BGvs{\varsigma}
\bmdefine\BGt{\tau}
\bmdefine\BGj{\tau}
\bmdefine\BGu{\upsilon}
\bmdefine\BGo{\omega}
\bmdefine\BGx{\xi}
\bmdefine\BGy{\psi}
\bmdefine\BGz{\zeta}
\bmdefine\BGD{\Delta}
\bmdefine\BGF{\Phi}
\bmdefine\BGG{\Gamma}
\bmdefine\BGL{\Lambda}
\bmdefine\BGP{\Pi}
\bmdefine\BGT{\Theta}
\bmdefine\BGS{\Sigma}
\bmdefine\BGU{\Upsilon}
\bmdefine\BGO{\Omega}
\bmdefine\BGX{\Xi}
\bmdefine\BGY{\Psi}

\bmdefine\BCA{{\mathcal A}}
\bmdefine\BCB{{\mathcal B}}
\bmdefine\BCC{{\mathcal C}}
\bmdefine\BCD{{\mathcal D}}
\bmdefine\BCE{{\mathcal E}}
\bmdefine\BCF{{\mathcal F}}
\bmdefine\BCG{{\mathcal G}}
\bmdefine\BCH{{\mathcal H}}
\bmdefine\BCI{{\mathcal I}}
\bmdefine\BCJ{{\mathcal J}}
\bmdefine\BCK{{\mathcal K}}
\bmdefine\BCL{{\mathcal L}}
\bmdefine\BCM{{\mathcal M}}
\bmdefine\BCN{{\mathcal N}}
\bmdefine\BCO{{\mathcal O}}
\bmdefine\BCP{{\mathcal P}}
\bmdefine\BCQ{{\mathcal Q}}
\bmdefine\BCR{{\mathcal R}}
\bmdefine\BCS{{\mathcal S}}
\bmdefine\BCT{{\mathcal T}}
\bmdefine\BCU{{\mathcal U}}
\bmdefine\BCV{{\mathcal V}}
\bmdefine\BCW{{\mathcal W}}
\bmdefine\BCX{{\mathcal X}}
\bmdefine\BCY{{\mathcal Y}}
\bmdefine\BCZ{{\mathcal Z}}

\bmdefine\Bzr{ 0}
\bmdefine\Ba{ a}
\bmdefine\Bb{ b}
\bmdefine\Bc{ c}
\bmdefine\Bd{ d}
\bmdefine\Be{ e}
\bmdefine\Bf{ f}
\bmdefine\Bg{ g}
\bmdefine\Bh{ h}
\bmdefine\Bi{ i}
\bmdefine\Bj{ j}
\bmdefine\Bk{ k}
\bmdefine\Bl{ l}
\bmdefine\Bm{ m}
\bmdefine\Bn{ n}
\bmdefine\Bo{ o}
\bmdefine\Bp{ p}
\bmdefine\Bq{ q}
\bmdefine\Br{ r}
\bmdefine\Bs{ s}
\bmdefine\Bt{ t}
\bmdefine\Bu{ u}
\bmdefine\Bv{ v}
\bmdefine\Bw{ w}
\bmdefine\Bx{ x}
\bmdefine\By{ y}
\bmdefine\Bz{ z}
\bmdefine\BA{ A}
\bmdefine\BB{ B}
\bmdefine\BC{ C}
\bmdefine\BD{ D}
\bmdefine\BE{ E}
\bmdefine\BF{ F}
\bmdefine\BG{ G}
\bmdefine\BH{ H}
\bmdefine\BI{ I}
\bmdefine\BJ{ J}
\bmdefine\BK{ K}
\bmdefine\BL{ L}
\bmdefine\BM{ M}
\bmdefine\BN{ N}
\bmdefine\BO{ O}
\bmdefine\BP{ P}
\bmdefine\BQ{ Q}
\bmdefine\BR{ R}
\bmdefine\BS{ S}
\bmdefine\BT{ T}
\bmdefine\BU{ U}
\bmdefine\BV{ V}
\bmdefine\BW{ W}
\bmdefine\BX{ X}
\bmdefine\BY{ Y}
\bmdefine\BZ{ Z}

